\documentclass[12pt,a4paper]{article}
\usepackage[cp1250]{inputenc}
\usepackage[english]{babel}
\usepackage{amsmath,amssymb,amsthm,amsfonts,MnSymbol}
\usepackage{tikz}
\usetikzlibrary{tikzmark}
\usepackage{mathtools}
\usepackage{bussproofs}

\author{Yury Savateev
\\
\normalsize{\textit{National
Research University Higher School of Economics}}\\
\normalsize{\textit{yury.savateev@gmail.com}}\\
\and
Daniyar Shamkanov
\\ 
\normalsize{\textit{Steklov Mathematical Institute of the Russian Academy of Sciences}}\\
\normalsize{\textit{National
Research University Higher School of Economics}}\\
\normalsize{\textit{daniyar.shamkanov@gmail.com}}\\
}
\title{Non-Well-Founded Proofs for the Grzegorczyk Modal Logic}
\date{}

\newtheorem{thm}{Theorem}[section]
\newtheorem{prop}[thm]{Proposition}
\newtheorem{lem}[thm]{Lemma}

\theoremstyle{remark}

\EnableBpAbbreviations

\begin{document}
\maketitle

\begin{abstract}
We present a sequent calculus for the Grzegorczyk modal logic $\mathsf{Grz}$ allowing cyclic and other non-well-founded proofs and obtain the cut-elimination theorem for it by constructing a continuous cut-elimination mapping acting on these proofs. As an application, we establish the Lyndon interpolation property for the logic $\mathsf{Grz}$ proof-theoretically.
\end{abstract}

\section{Introduction}
\label{s1}
A non-well-founded proof is usually defined as a possibly infinite tree of formulas (sequents) that is constructed according to inference rules of a proof system and, in addition, that satisfies a particular condition on infinite branches. A cyclic, or circular, proof can be defined as a finite pointed graph of formulas (sequents) which unraveling yields a non-well-founded proof. These proofs turn out to be an interesting alternative to traditional proofs for logics with inductive and co-inductive definitions, fixed-point operators and similar features. For example, proof systems allowing cyclic proofs can be defined for the modal $\mu$-calculus \cite{AfLe}, the Lambek calculus with iteration \cite{Kuz} and for Peano arithmetic \cite{Sim}. In the last case, these proofs can be understood as a formalization of the concept of proof by infinite descend.

Structural proof theory of deductive systems allowing cyclic and non-well-founded proofs seems to be underdeveloped. In \cite{ForSan}, J.~Fortier and L.~Santocanale considered the case of the $\mu$-calculus with additive connectives. They present a procedure eliminating all applications of the cut-rule from a cyclic proof and resulting an infinite tree of sequents. Though they don't show that this tree satisfies a guard condition on infinite branches, which is necessary for non-well-founded proofs in the $\mu$-calculus. Unfortunately, we don't know other syntactic cut-elimination results for systems with non-well-founded proofs.    
Here we present a sequent calculus for the Grzegorczyk modal logic allowing non-well-founded proofs and obtain the cut-elimination theorem for it. This article is an extended version of the conference paper \cite{SavSham}.

The Grzegorczyk modal logic $\mathsf{Grz}$ is a well-known modal logic \cite{Maks}, which can be characterized by reflexive partially ordered Kripke frames without infinite ascending chains. This logic is complete w.r.t. the arithmetical semantics, where the modal connective $\Box$ corresponds to the strong provability operator \textit{"... is true and provable"} in Peano arithmetic. There is a translation from $\mathsf{Grz}$ into the G\"{o}del-L\"{o}b provability logic $\mathsf{GL}$ such that 
$$\mathsf{Grz} \vdash A \Longleftrightarrow \mathsf{GL} \vdash A^\ast,$$ where $A^\ast$ is obtained from $A$ by replacing all subformulas of the form $\Box B$ by $B \wedge \Box B$.

Recently a new proof-theoretic presentation for the G\"{o}del-L\"{o}b provability logic $\mathsf{GL}$ in the form of a sequent calculus  allowing non-well-founded proofs was given in \cite{Sham, Iemhoff}. 
We wonder whether cyclic and, more generally, non-well-founded proofs can be fruitfully considered in the case of $\mathsf{Grz}$. 
We consider a sequent calculus allowing non-well-founded proofs for the Grzegorczyk modal logic and present the cut-elimination theorem for the given system.
In order to avoid nested co-inductive and inductive reasoning, we adopt an approach from denotational semantics of computer languages, where program types are interpreted as ultrametric spaces and fixed-point combinators are encoded using the Banach fixed-point theorem (see \cite{BaMa}, \cite{Esc}, \cite{Breu}). We consider the set of non-well-founded proofs of $\mathsf{Grz}$ and various sets of operations acting on theses proofs as ultrametric spaces and define our cut-elimination operator using the Prie\ss-Crampe fixed-point theorem (see \cite{PrCr}), which is a strengthening of the Banach's theorem. 
In this paper, we also establish that in our sequent system it is sufficient to consider only unravellings of cyclic proofs instead of arbitrary non-well-founded proofs in order to obtain all provable sequents. As an application of cut-elimination, we obtain the Lyndon interpolation property for the Grzegorczyk modal logic $\mathsf{Grz}$ proof-theoretically.


Recall that the Craig interpolation property for a logic $\mathsf{L}$ says that if $A$ implies $B$, then there is an interpolant, that is, a formula $I$ containing only common variables of $A$ and $B$ such that $A$ implies $I$ and $I$ implies $B$. The Lyndon interpolation property is a strengthening of the Craig one that also takes into consideration negative and positive occurrences of the shared propositional variables; that is, the variables occurring in $I$ positively (negatively) must also occur both in $A$ and $B$ positively (negatively).

Though the Grzegorczyk logic has the Lyndon interpolation property \cite{Maks2}, there were seemingly no syntactic proofs for this result.  
It is unclear how Lyndon interpolation can be obtained from previously introduced sequent systems for $\mathsf{Grz}$ \cite{Avron, BorGen, DyNe} by direct proof-theoretic arguments because these systems contain inference rules in which a polarity
change occurs under the passage from the principal formula in the conclusion to its immediate ancestors in the premise. In this article, we give a syntactic proof of Lyndon interpolation for the Grzegorczyk modal logic as an application of our cut-elimination theorem.
  
The paper is organised as follows. In Section \ref{SecPrel} we recall a standart sequent calculus for $\mathsf{Grz}$. In Section \ref{SecNWF} we introduce the proof system $\mathsf{Grz_{\infty}}$ that allows non-well-founded proofs and prove its equivalence to the standard one. In Section \ref{SecUlt} we recall basic notions of the theory of ultrametric spaces and consider several relevant examples. In Section \ref{SecAdm} we state admissability of several rules for our system that will be used later. In Section \ref{SecCut} we establish the cut elimination result for the system $\mathsf{Grz}_\infty$ syntactically. In \ref{SecLyn} we prove the Lyndon interpolation property for the logic $\mathsf{Grz}$. Finally, in Section \ref{SecCyc} we establish that every provable sequent of $\mathsf{Grz}_\infty$ has a cyclic proof.

\section{Preliminaries}
\label{SecPrel}
In this section we recall the Grzegorczyk modal logic $\mathsf{Grz}$ and define an ordinary sequent calculus for it. 

\textit{Formulas} of $\mathsf{Grz}$, denoted by $A$, $B$, $C$, are built up as follows:
$$ A ::= \bot \,\,|\,\, p \,\,|\,\, (A \to A) \,\,|\,\, \Box A \;, $$
where $p$ stands for atomic propositions. 
We treat other boolean connectives and the modal operator $\Diamond$ as abbreviations:
\begin{gather*}
\neg A := A\to \bot,\qquad\top := \neg \bot,\qquad A\wedge B := \neg (A\to \neg B),
\\
A\vee B := (\neg A\to B),\qquad\Diamond A := \neg\Box \neg A.
\end{gather*}


The Hilbert-style axiomatization of $\mathsf{Grz}$ is given by the following axioms and inference rules:

\textit{Axioms:}
\begin{itemize}
\item[(i)] Boolean tautologies;
\item[(ii)] $\Box (A \rightarrow B) \rightarrow (\Box A \rightarrow \Box B)$;
\item[(iii)] $\Box A \rightarrow \Box \Box A$;
\item[(iv)] $\Box A \rightarrow A$;
\item[(v)] $\Box(\Box(A \rightarrow \Box A) \rightarrow A) \rightarrow \Box A$.
\end{itemize}

\textit{Rules:} modus ponens, $A / \Box A$. \\


Now we define an ordinary sequent calculus for $\mathsf{Grz}$. A \textit{sequent} is an expression of the form $\Gamma \Rightarrow \Delta$, where $\Gamma$ and~$\Delta$ are finite multisets of formulas. For a multiset of formulas $\Gamma = A_1,\dotsc, A_n$, we define $\Box \Gamma$ to be $\Box A_1,\dotsc, \Box A_n$.


The system $\mathsf{Grz_{Seq}}$, which is a variant of the sequent calculus from \cite{Avron}, is defined by the following initial sequents and inference rules: 
\begin{gather*}
\AXC{ $\Gamma, A \Rightarrow A, \Delta $ ,}
\DisplayProof \qquad
\AXC{ $\Gamma , \bot \Rightarrow \Delta $ ,}
\DisplayProof\\\\
\AXC{$\Gamma , B \Rightarrow \Delta $}
\AXC{$\Gamma \Rightarrow A,\Delta $}
\LeftLabel{$\mathsf{\to_L}$}
\BIC{$\Gamma , A \to B \Rightarrow \Delta$}
\DisplayProof\;,\qquad
\AXC{$\Gamma , A \Rightarrow B, \Delta $}
\LeftLabel{$\mathsf{\to_R}$}
\UIC{$\Gamma \Rightarrow A \to B, \Delta$}
\DisplayProof\;,\\\\
\AXC{$\Gamma, B, \Box B \Rightarrow \Delta $}
\LeftLabel{$\mathsf{refl}$}
\UIC{$\Gamma , \Box B \Rightarrow \Delta$}
\DisplayProof\;,\qquad
\AXC{$ \Box\Pi, \Box(A\to\Box A) \Rightarrow A$}
\LeftLabel{$\mathsf{\Box_{Grz}}$}
\UIC{$\Gamma, \Box \Pi \Rightarrow \Box A ,\Delta $}
\DisplayProof \;.
\end{gather*}
\begin{center}
\textbf{Fig. 1.} The system $\mathsf{Grz_{Seq}}$
\end{center}

The cut rule has the form
\begin{gather*}
\AXC{$\Gamma\Rightarrow A,\Delta$}
\AXC{$\Gamma,A\Rightarrow\Delta$}
\LeftLabel{$\mathsf{cut}$}
\RightLabel{ ,}
\BIC{$\Gamma\Rightarrow\Delta$}
\DisplayProof
\end{gather*}

where $A$ is called the \emph{cut formula} of the given inference. 

\begin{lem} \label{prop}
$\mathsf{Grz_{Seq}} + \mathsf{cut}\vdash \Gamma\Rightarrow\Delta$ if and only if $\mathsf{Grz} \vdash \bigwedge\Gamma\to\bigvee\Delta $. 
\end{lem}
This lemma is completely standard, so we omit the proof.

\begin{thm}\label{cutelimgrz}
If $\mathsf{Grz_{Seq}} + \mathsf{cut}\vdash \Gamma\Rightarrow\Delta$, then $\mathsf{Grz_{Seq}} \vdash \Gamma\Rightarrow\Delta$. 

\end{thm}

A syntactic cut-elimination for $\mathsf{Grz}$ was obtained by M.~Borga and P.~Gentilini in \cite{BorGen}. In this paper, we will give another proof of this cut-elimination theorem in the next sections.




\section{Non-well-founded proofs}
\label{SecNWF}
In this section we introduce a sequent calculus for $\mathsf{Grz}$ allowing non-well-founded proofs and define two translations that connect ordinary and non-well-founded sequent systems. 

Inference rules and initial sequents of the sequent calculus $\mathsf{Grz_\infty}$ have the following form:
\begin{gather*}
\AXC{ $\Gamma, p \Rightarrow p, \Delta $ ,}
\DisplayProof\qquad
\AXC{ $\Gamma , \bot \Rightarrow  \Delta$ ,}
\DisplayProof \\\\
\AXC{$\Gamma , B \Rightarrow \Delta $}
\AXC{$\Gamma \Rightarrow A,\Delta $}
\LeftLabel{$\mathsf{\to_L}$}
\BIC{$\Gamma , A \to B \Rightarrow \Delta$}
\DisplayProof\;,\qquad
\AXC{$\Gamma , A \Rightarrow B, \Delta $}
\LeftLabel{$\mathsf{\to_R}$}
\UIC{$\Gamma \Rightarrow A \to B, \Delta$}
\DisplayProof\;,\\\\
\AXC{$\Gamma, B, \Box B \Rightarrow \Delta $}
\LeftLabel{$\mathsf{refl}$}
\UIC{$\Gamma , \Box B \Rightarrow \Delta$}
\DisplayProof\;,\qquad
\AXC{$\Gamma, \Box \Pi \Rightarrow A, \Delta$}
\AXC{$\Box \Pi \Rightarrow A$}
\LeftLabel{$\mathsf{\Box}$}
\BIC{$\Gamma, \Box \Pi \Rightarrow \Box A, \Delta$}
\DisplayProof \;.
\end{gather*}
\begin{center}
\textbf{Fig. 2.} The system $\mathsf{Grz}_\infty$
\end{center}

The system $\mathsf{Grz}_{\infty}+\mathsf{cut}$ is defined by adding the rule ($\mathsf{cut}$) to the system $\mathsf{Grz_\infty}$.
An \emph{$\infty$--proof} in $\mathsf{Grz}_\infty$ ($\mathsf{Grz}_{\infty}+\mathsf{cut}$) is a (possibly infinite) tree whose nodes are marked by
sequents and whose leaves are marked by initial sequents and that is constructed according to the rules of the sequent calculus. In addition, every infinite branch in an $\infty$--proof must pass through a right premise of the rule ($\Box$) infinitely many times. A sequent $\Gamma \Rightarrow \Delta$ is \emph{provable} in $\mathsf{Grz}_\infty$ ($\mathsf{Grz}_{\infty}+\mathsf{cut}$) if there is an $\infty$--proof in $\mathsf{Grz}_\infty$ ($\mathsf{Grz}_{\infty}+\mathsf{cut}$) with the root marked by $\Gamma \Rightarrow \Delta$.


The \emph{$n$-fragment} of an $\infty$--proof is a finite tree obtained from the $\infty$--proof by cutting every branch at the $n$th from the root right premise of the rule ($\Box$). The $1$-fragment of an $\infty$--proof is also called its \emph{main fragment}. We define the \emph{local height $\lvert \pi \rvert$ of an $\infty$--proof $\pi$} as the length of the longest branch in its main fragment. An $\infty$--proof only consisting of an initial sequent has height 0.

For instance, consider an $\infty$--proof of the sequent $\Box(\Box(p \rightarrow \Box p) \rightarrow p) \Rightarrow p$: 

\begin{gather*}
\AXC{\textsf{Ax}}
\noLine
\UIC{$ F, p\Rightarrow p$}
\AXC{\textsf{Ax}}
\noLine
\UIC{$ F,p\Rightarrow \Box p,p$}
\LeftLabel{$\mathsf{\to_R}$}
\UIC{$ F \Rightarrow p\to\Box p,p$}
\AXC{\textsf{Ax}}
\noLine
\UIC{$p, F \Rightarrow p$}
\AXC{$\vdots$}
\noLine
\UIC{$ F \Rightarrow p$} 
\LeftLabel{$\mathsf{\Box}$} 
\BIC{$p, F \Rightarrow \Box p$}
\LeftLabel{$\mathsf{}\to_R$}
\UIC{$ F \Rightarrow p\to\Box p$}
\LeftLabel{$\mathsf{\Box}$} 
\BIC{$ F \Rightarrow \Box(p\to \Box p),p$} 
\LeftLabel{$\mathsf{\to_L}$}
\BIC{$\Box(p \rightarrow \Box p) \rightarrow p, F \Rightarrow p$}
\LeftLabel{$\mathsf{refl}$}
\RightLabel{ ,} 
\UIC{$F \Rightarrow p $}
\DisplayProof
\end{gather*}
where $F=\Box(\Box(p \rightarrow \Box p) \rightarrow p) $. 
The local height of this $\infty$--proof equals to 4 and its main fragment has the form
\begin{gather*}
\AXC{\textsf{Ax}}
\noLine
\UIC{$ F, p\Rightarrow p$}
\AXC{\textsf{Ax}}
\noLine
\UIC{$ F,p\Rightarrow \Box p,p$}
\LeftLabel{$\mathsf{\to_R}$}
\UIC{$ F \Rightarrow p\to\Box p,p$}
\AXC{\qquad \qquad \qquad \qquad}
\LeftLabel{$\mathsf{\Box}$} 
\BIC{$ F \Rightarrow \Box(p\to \Box p),p$} 
\LeftLabel{$\mathsf{\to_L}$}
\BIC{$\Box(p \rightarrow \Box p) \rightarrow p, F \Rightarrow p$}
\LeftLabel{$\mathsf{refl}$}
\RightLabel{ .} 
\UIC{$F \Rightarrow p $}
\DisplayProof
\end{gather*}
We denote the set of all $\infty$-proofs in the system $\mathsf{Grz}_{\infty} +\mathsf{cut} $ by $\mathcal P$. 
For $\pi, \tau\in\mathcal P$, we write $\pi \sim_n \tau$ if $n$-fragments of these $\infty$-proofs are coincide. For any $\pi, \tau\in\mathcal P$, we also set $\pi \sim_0 \tau$.

Now we define two translations that connect ordinary and non-well-founded sequent calculi for $\mathsf{Grz}$. 

\begin{lem}\label{AtoA}
We have  $\mathsf{Grz}_{\infty} \vdash \Gamma,A\Rightarrow A,\Delta$ for any sequent $\Gamma \Rightarrow \Delta$ and any formula $A$.
\end{lem}
\begin{proof}
Standard induction on the structure of $A$.
\end{proof}

\begin{lem}\label{Grz-schema}
We have $\mathsf{Grz}_{\infty}\vdash\Box(\Box(A \rightarrow \Box A) \rightarrow A) \Rightarrow A$ for any formula $A$.
\end{lem}
\begin{proof}
Consider an example of $\infty$--proof for the sequent $\Box(\Box(p \rightarrow \Box p) \rightarrow p) \Rightarrow p$ given above. We transform this example into an $\infty$--proof for $\Box(\Box(A \rightarrow \Box A) \rightarrow A) \Rightarrow A$ by replacing $p$ with $A$ and adding required $\infty$--proofs instead of initial sequents using Lemma \ref{AtoA}.  
\end{proof}

Recall that an inference rule is called admissible (in a given proof system) if, for any instance of the rule, the conclusion is provable whenever all premises are provable. 
\begin{lem}\label{weakening}
The rule
\begin{gather*}
\AXC{$\Gamma\Rightarrow\Delta$}
\LeftLabel{$\mathsf{weak}$}
\UIC{$\Pi,\Gamma\Rightarrow\Delta,\Sigma$}
\DisplayProof
\end{gather*}
is admissible in the systems $\mathsf{Grz_{Seq}} $ and $\mathsf{Grz}_{\infty} +\mathsf{cut}$.
\end{lem}
\begin{proof}
Standard induction on the structure (local height) of a proof of $\Gamma\Rightarrow\Delta$.
\end{proof}

\begin{thm}\label{seqtoinfcut}
If $\mathsf{Grz_{Seq}}+\mathsf{cut}\vdash\Gamma\Rightarrow\Delta$, then $\mathsf{Grz}_{\infty}+\mathsf{cut}\vdash\Gamma\Rightarrow\Delta$.
\end{thm}
\begin{proof}
Assume $\pi$ is a proof of $\Gamma\Rightarrow\Delta$ in $\mathsf{Grz_{Seq}}+\mathsf{cut}$. By induction on the size of $\pi$ we prove $\mathsf{Grz}_{\infty}+\mathsf{cut}\vdash\Gamma\Rightarrow\Delta$. 

If $\Gamma \Rightarrow \Delta $ is an initial sequent of $\mathsf{Grz_{Seq}}+\mathsf{cut}$, then it is provable in $\mathsf{Grz}_{\infty}+\mathsf{cut}$ by Lemma \ref{AtoA}.
Otherwise, consider the last application of an inference rule in $\pi$. 

The only non-trivial case is when the proof $\pi$ has the form 
\begin{gather*}
\AXC{$\pi^\prime$}
\noLine
\UIC{$\Box \Pi,\Box(A\to\Box A)\Rightarrow A$}
\LeftLabel{$\mathsf{\Box_{Grz}}$}
\RightLabel{ ,}
\UIC{$\Sigma,\Box\Pi\Rightarrow \Box A, \Lambda$}
\DisplayProof
\end{gather*} 
where $\Sigma,\Box\Pi = \Gamma$ and $\Box A, \Lambda = \Delta$. By the induction hypothesis there is an $\infty$--proof $\xi$ of $\Box \Pi,\Box(A\to\Box A)\Rightarrow A$ in $\mathsf{Grz}_{\infty}+\mathsf{cut}$.

We have the following $\infty$--proof $\lambda$ of $\Box \Pi\Rightarrow A$ in $\mathsf{Grz}_{\infty}+\mathsf{cut} $:
\begin{gather*}
\AXC{$\xi^{\prime}$}
\noLine
\UIC{$\Box \Pi,\Box(A\to\Box A)\Rightarrow A,A$}
\LeftLabel{$\mathsf{\to_R}$}
\UIC{$\Box \Pi\Rightarrow G,A$}
\AXC{$\xi$}
\noLine
\UIC{$\Box \Pi,\Box(A\to\Box A)\Rightarrow A$}
\LeftLabel{$\mathsf{\to_R}$}
\UIC{$\Box \Pi\Rightarrow G$}
\LeftLabel{$\Box$}
\BIC{$\Box \Pi\Rightarrow\Box G,A$}
\AXC{$\theta$}
\noLine
\UIC{$\Box\Pi,\Box G \Rightarrow A$}
\LeftLabel{$\mathsf{cut}$}
\RightLabel{ ,}
\BIC{$\Box\Pi\Rightarrow A$}
\DisplayProof
\end{gather*}
where $G= \Box(A \rightarrow \Box A) \rightarrow A$, $\xi^{\prime}$ is an $\infty$--proof of $\Box \Pi,\Box(A\to\Box A)\Rightarrow A,A$ obtained from $\xi$ by Lemma \ref{weakening} and $\theta$ is an $\infty$--proof of $\Box\Pi,\Box G \Rightarrow A$, which exists by Lemma \ref{Grz-schema} and Lemma \ref{weakening}.

The required $\infty$--proof for $\Sigma,\Box\Pi\Rightarrow \Box A, \Delta$ has the form
\begin{gather*}
\AXC{$\lambda'$}
\noLine
\UIC{$\Sigma,\Box\Pi\Rightarrow A,\Lambda$}
\AXC{$\lambda$}
\noLine
\UIC{$\Box\Pi\Rightarrow A$}
\LeftLabel{$\Box$}
\RightLabel{ ,}
\BIC{$\Sigma,\Box\Pi\Rightarrow \Box A, \Lambda$}
\DisplayProof
\end{gather*}
where $\lambda'$ is an $\infty$-proof of the sequent $\Sigma,\Box\Pi\Rightarrow A,\Lambda$ obtained from the $\infty$-proof $\lambda$ by Lemma \ref{weakening}.

The cases of other inference rules being last in $\pi$ are straightforward, so we omit them.

\end{proof}


For a sequent $\Gamma\Rightarrow\Delta$, let $Sub(\Gamma\Rightarrow\Delta)$ be the set of all subformulas of the formulas from $\Gamma \cup\Delta$.
For a finite set of formulas $\Lambda$, let $\Lambda^\ast$ be the set $\{\Box(A\to\Box A)\mid A\in\Lambda\}$.

\begin{lem} \label{translation}
If $\mathsf{Grz_\infty}\vdash \Gamma\Rightarrow\Delta$, then $\mathsf{Grz_{Seq}} \vdash \Lambda^\ast,\Gamma\Rightarrow\Delta$ for any finite set of formulas $\Lambda$.
\end{lem}
\begin{proof}
Assume $\pi$ is an $\infty$--proof of the sequent $\Gamma\Rightarrow\Delta$ in $\mathsf{Grz}_\infty$ and $\Lambda$ is a finite set of formulas.
By induction on the number of elements in the finite set $Sub(\Gamma\Rightarrow\Delta)\backslash \Lambda$ with a subinduction on $\lvert \pi \rvert$, we prove $\mathsf{Grz_{Seq}} \vdash \Lambda^\ast,\Gamma\Rightarrow\Delta$. 


If $\lvert \pi \rvert=0$, then $\Gamma\Rightarrow\Delta$ is an initial sequent. We see that the sequent $\Lambda^\ast,\Gamma\Rightarrow\Delta$ is an initial sequent and it is provable in $\mathsf{Grz_{Seq}}$.
Otherwise, consider the last application of an inference rule in $\pi$. 

Case 1. Suppose that $\pi$ has the form
\begin{gather*}
\AXC{$\pi^\prime$}
\noLine
\UIC{$\Gamma,A\Rightarrow B,\Sigma$}
\LeftLabel{$\mathsf{\to_R}$}
\RightLabel{ ,}
\UIC{$\Gamma\Rightarrow A\to B,\Sigma$}
\DisplayProof
\end{gather*}
where $A\to B,\Sigma = \Delta$.
Notice that $\lvert \pi^\prime \rvert < \lvert \pi \rvert $. By the induction hypothesis for $\pi^\prime$ and $\Lambda$, the sequent $\Lambda^\ast,\Gamma,A\Rightarrow B,\Sigma$ is provable in  $\mathsf{Grz_{Seq}}$. 
Applying the rule ($\mathsf{\to_R}$) to it, we obtain that the sequent $\Lambda^\ast,\Gamma\Rightarrow\Delta$ is provable in $\mathsf{Grz_{Seq}}$.

Case 2. Suppose that $\pi$ has the form
\begin{gather*}
\AXC{$\pi^\prime$}
\noLine
\UIC{$\Sigma, B\Rightarrow \Delta$}
\AXC{$\pi^{\prime\prime}$}
\noLine
\UIC{$\Sigma \Rightarrow A,\Delta$}
\LeftLabel{$\mathsf{\to_L}$}
\RightLabel{ ,}
\BIC{$\Sigma, A\to B\Rightarrow \Delta$}
\DisplayProof
\end{gather*}
where $\Sigma, A\to B = \Gamma$. We see that $\lvert \pi^\prime \rvert < \lvert \pi \rvert $. By the induction hypothesis for $\pi^\prime$ and $\Lambda$, the sequent $\Lambda^\ast,\Sigma, B\Rightarrow \Delta$ is provable in  $\mathsf{Grz_{Seq}}$. Analogously, we have $\mathsf{Grz_{Seq}} \vdash \Lambda^\ast,\Sigma \Rightarrow A,\Delta$. Applying the rule ($\mathsf{\to_L}$), we obtain that the sequent $\Lambda^\ast,\Sigma, A\to B \Rightarrow\Delta$ is provable in $\mathsf{Grz_{Seq}}$.

Case 3. Suppose that $\pi$ has the form
\begin{gather*}
\AXC{$\pi^\prime$}
\noLine
\UIC{$\Sigma,A,\Box A\Rightarrow \Delta$}
\LeftLabel{$\mathsf{refl}$}
\RightLabel{ ,}
\UIC{$\Sigma,\Box A\Rightarrow \Delta$}
\DisplayProof
\end{gather*}
where $\Sigma, \Box A = \Gamma$. We see that $\lvert \pi^\prime \rvert < \lvert \pi \rvert $. By the induction hypothesis for $\pi^\prime$ and $\Lambda$, the sequent $\Lambda^\ast,\Sigma, A, \Box A\Rightarrow \Delta$ is provable in  $\mathsf{Grz_{Seq}}$. Applying the rule ($\mathsf{refl}$), we obtain $\mathsf{Grz_{Seq}} \vdash \Lambda^\ast,\Sigma, \Box A\Rightarrow\Delta$. 

Case 4. Suppose that $\pi$ has the form
\begin{gather*}
\AXC{$\pi^\prime$}
\noLine
\UIC{$\Phi, \Box \Pi \Rightarrow A, \Sigma$}
\AXC{$\pi^{\prime\prime}$}
\noLine
\UIC{$\Box \Pi \Rightarrow A$}
\LeftLabel{$\mathsf{\Box}$}
\RightLabel{ ,}
\BIC{$\Phi, \Box \Pi \Rightarrow \Box A, \Sigma$}
\DisplayProof
\end{gather*}
where $\Phi, \Box \Pi = \Gamma$ and $\Box A, \Sigma =\Delta$.

Subcase 4.1: the formula $A$ belongs to $\Lambda$. We see that $\lvert \pi^\prime \rvert < \lvert \pi \rvert $. By the induction hypothesis for $\pi^\prime$ and $\Lambda$, the sequent $\Lambda^\ast,\Phi, \Box \Pi \Rightarrow A, \Sigma$ is provable in  $\mathsf{Grz_{Seq}}$. 
Then we see
\begin{gather*}
\AXC{$\mathsf{Ax}$}
\noLine
\UIC{$\Lambda^\ast,\Box A,\Phi, \Box \Pi \Rightarrow \Box A, \Sigma$}
\AXC{$\Lambda^\ast,\Phi, \Box \Pi \Rightarrow A,  \Sigma$}
\LeftLabel{$\mathsf{weak}$}
\UIC{$\Lambda^\ast,\Phi, \Box \Pi \Rightarrow A,  \Box A,\Sigma$}
\LeftLabel{$\mathsf{\to_L}$}
\BIC{$(\Lambda\backslash\{A\})^\ast,A\to\Box A,\Box(A\to\Box A),\Phi, \Box \Pi \Rightarrow \Box A, \Sigma$}
\LeftLabel{$\mathsf{refl}$}
\RightLabel{ ,}
\UIC{$(\Lambda\backslash\{A\})^\ast,\Box(A\to\Box A),\Phi, \Box \Pi \Rightarrow \Box A, \Sigma$}
\DisplayProof
\end{gather*}
where the rule ($\mathsf{weak}$) is admissible by Lemma \ref{weakening}.

Subcase 4.2: the formula $A$ doesn't belong to $\Lambda$. We have that the number of elements in $Sub(\Box\Pi\Rightarrow A)\backslash(\Lambda\cup \{A\})$ is strictly less than the number of elements in $Sub(\Phi, \Box \Pi \Rightarrow \Box A, \Sigma)\backslash\Lambda$. Therefore, by the induction hypothesis for $\pi^{\prime\prime}$ and $\Lambda\cup \{A\}$, the sequent $\Lambda^\ast,\Box(A\to\Box A),\Box \Pi \Rightarrow A$ is provable in  $\mathsf{Grz_{Seq}}$. Then we have
\begin{gather*}
\AXC{$\Lambda^\ast,\Box(A\to\Box A),\Box \Pi \Rightarrow A$}
\LeftLabel{$\mathsf{\Box_{Grz}}$}
\RightLabel{ .}
\UIC{$\Lambda^\ast,\Phi, \Box \Pi \Rightarrow \Box A, \Sigma$}
\DisplayProof
\end{gather*}
\end{proof}
From Lemma \ref{translation} we immediately obtain the following theorem.
\begin{thm}\label{inftoseq}
If $\mathsf{Grz_\infty}\vdash \Gamma\Rightarrow\Delta$, then $\mathsf{Grz_{Seq}} \vdash \Gamma\Rightarrow\Delta$.
\end{thm}

\section{Ultrametric spaces}
\label{SecUlt}
In this section we recall basic notions of the theory of ultrametric spaces (cf. \cite{Shor}) and consider several examples concerning $\infty$-proofs.

An \emph{ultrametric space} $(M,d)$ is a metric space that satisfies a stronger version of the triangle inequality: for any $x,y,z\in M$
$$d(x,z) \leqslant \max \{d(x,y), d(y,z)\}.$$ 

For $x\in M$ and $r\in [0, + \infty)$, the set $B_r(x)=\{ y\in M \mid d(x,y) \leqslant r\}$ is called the \emph{closed ball} with \emph{center $x$} and \emph{radius $r$}. Recall that a metric space $(M,d)$ is \emph{complete} if any descending sequence of closed balls, with radii tending to $0$, has a common point. An ultametric space $(M,d)$ is called \emph{spherically complete} if an arbitrary descending sequence of closed balls has a common point.

For example, consider the set $\mathcal P$ of all $\infty$-proofs of the system $\mathsf{Grz}_{\infty} +\mathsf{cut} $. We can define an ultrametric $d_{\mathcal P} \colon {\mathcal P} \times {\mathcal P} \to [0,1]$ on $\mathcal P$ by putting
\[d_{\mathcal P}(\pi,\tau) = 
\inf\{\frac{1}{2^n} \mid  \pi \sim_n \tau\}.
\]
We see that $d_{\mathcal P}(\pi , \tau) \leqslant 2^{-n}$ if and only if $\pi \sim_n \tau$. Thus, the ultrametric $d_{\mathcal P}$ can be considered as a measure of similarity between $\infty$-proofs.

\begin{prop}\label{ComplP}
$(\mathcal{P},d_{\mathcal P})$ is a (spherically) complete ultrametric space.
\end{prop}

Consider the following characterization of spherically complete ultrametric spaces. Let us write $x \equiv_r y$ if $d(x,y)\leqslant r$.
Trivially, the relation $\equiv_r$ is an equivalence relation for any ultrametric space and any number $r\geqslant 0$. 

\begin{prop}\label{sphcomp}
An ultametric space $(M,d)$ is spherically complete if and only if for any sequence $(x_i)_{i\in \mathbb N}$ of elements of $M$, where $x_i\equiv_{r_i}x_{i+1}$ and  
$r_i \geqslant r_{i+1}$ for all $i\in \mathbb N$, there is a point $x$ of $M$ such that $x \equiv_{r_i} x_i$ for any $i\in\mathbb N$.
\end{prop}
\begin{proof}
($\Rightarrow$) Assume $(M,d)$ is a spherically complete ultrametric space. Consider a sequence $(x_i)_{i\in \mathbb N}$ of elements of $M$ such that $x_i\equiv_{r_i}x_{i+1}$ and $r_i \geqslant r_{i+1}$ for all $i\in \mathbb N$.  Then the sequence $(B_{r_i}(x_i))$ is a descending sequence of closed balls, and therefore by spherical completeness has a common point $x$. Trivially, the point $x$ satisfies the desired conditions.

($\Leftarrow$) Assume there is a descending sequence of closed balls $(B_{r_i}(x_i))$. We have that $x_0 \equiv_{r_0} x_1 \equiv_{r_1} \dots$ and $r_i \geqslant r_{i+1}$ for all $i\in \mathbb N$. So there is an element $x\in M$ such that $x \equiv_{r_i} x_i$, meaning it lies in all the balls.
\end{proof}

In an ultrametric space $(M,d)$, a function $f \colon M\to M$ is called \emph{non-expansive} if $d(f(x),f(y)) \leqslant d(x,y)$ for all $x,y\in M$.
For ultrametric spaces $(M, d_M)$ and $(N, d_N)$, the Cartesian product $M \times N$ can be also considered as an ultrametric space with the metric $$d_{M \times N} ((x_1,y_1),(x_2,y_2)) = \max \{d_M(x_1,x_2), d_N(y_1,y_2) \}.$$ 


Let us consider another example. For $m\in \mathbb N$, let $\mathcal F_m$ denote the set of all non-expansive functions from $\mathcal P^m$ to $\mathcal P$. Note that any function $\mathsf u\colon \mathcal P^m\to\mathcal P$ is non-expansive if and only if for any tuples $\vec\pi$ and $\vec\pi'$, and any $n\in\mathbb N$ we have 
$$\pi_1\sim_n\pi'_1,\dotsc,\pi_m\sim_n\pi'_m\Rightarrow \mathsf u(\vec\pi)\sim_n\mathsf u(\vec\pi').$$
Now we introduce an ultrametric for $\mathcal F_m$. For $\mathsf{a,b}\in \mathcal F_m$, we write 
$\mathsf a\sim_{n,k}\mathsf b$ if $\mathsf a(\vec{\pi})\sim_n\mathsf b(\vec{\pi})$ for any $\vec{\pi}\in\mathcal P^m$ and, in addition, $\mathsf a(\vec{\pi}) \sim_{n+1}\mathsf b(\vec{\pi})$ whenever $\Sigma_{i=1}^m\lvert\pi_i\rvert< k$.\footnote{This definition is inspired by \cite[Subsection 2.1]{GiMi}.}
An ultrametric $l_m$ on $\mathcal F_m$ is defined by 
\[l_m(\mathsf a,\mathsf b)=\frac{1}{2}\inf\{\frac{1}{2^n}+\frac{1}{2^{n+k}}\mid \mathsf a\sim_{n,k}\mathsf b\}.\]
We see that $l_m(\mathsf a,\mathsf b) \leqslant 2^{-n-1}+2^{-n-k-1}$ if and only if $\mathsf a\sim_{n,k}\mathsf b$.

\begin{prop}\label{SphCom}
$(\mathcal F_m, l_m)$ is a spherically complete ultrametric space.
\end{prop}
\begin{proof}
Assume we have a series
$\mathsf a_0\sim_{n_0,k_0}\mathsf a_1\sim_{n_1,k_1}\dotso  $,
where the sequence $r_i=2^{-n_i}+2^{-n_i-k_i-1}$ is non-increasing. From Proposition \ref{sphcomp}, it is sufficient to find a function $\mathsf a\in\mathcal F_m$ such that $\mathsf a\sim_{n_i,k_i}\mathsf a_i$ for all $i\in \mathbb N$.


Suppose $\lim_{i\to\infty}r_i=0$. Consider a tuple $\vec{\pi}\in\mathcal P^m $. We have that $\lim_{i\to\infty}n_i=+ \infty$ and $\mathsf a_0 (\vec{\pi})\sim_{n_0}\mathsf a_1 (\vec{\pi})\sim_{n_1}\dotso$ .
By Proposition \ref{ComplP}, there is an $\infty$-proof $\tau$ such that $\tau \sim_{n_i} \mathsf a_i (\vec{\pi})$ for all $i\in \mathbb N$. We define $\mathsf a (\vec{\pi}) = \tau$. We need to check that the mapping $\mathsf a$ is non-expansive. If for tuples $\vec\pi$ and $\vec\pi'$ we have $\pi_1\sim_n\pi'_1,\ldots,\pi_m\sim_n\pi'_m$ for some $n\in\mathbb N$, then we can choose $i$ such that $n_i>n$. We have $\mathsf a(\vec\pi)\sim_n \mathsf a_i(\vec\pi)\sim_n\mathsf a_i(\vec\pi')\sim_n\mathsf a(\vec\pi')$. Therefore $\mathsf a(\vec\pi)\sim_n \mathsf a(\vec\pi')$ and the mapping $\mathsf a$ is non-expansive.

If $\lim_{i\to\infty}r_i>0$, then $\lim_{i\to\infty}n_i= n$ for some $n \in \mathbb{N}$. We have two cases: either $\lim_{i\to\infty}k_i= k$ for a number $k \in \mathbb{N}$, or $\lim_{i\to\infty}k_i= +\infty$. In the first case, there is $j\in\mathbb N$ such that $(n_i,k_i)=(n_j,k_j)$ for all $i>j$, and we can take $\mathsf a_j$ as $\mathsf a$. Here the mapping $\mathsf a$ is obviously non-expansive.
In the second case, for a tuple $\vec{\pi}$ we define $\mathsf a (\vec{\pi})$ to be $\mathsf{a}_j (\vec{\pi})$, where $j= \min \{ i \in \mathbb{N} \mid n_i= n \text{ and } \sum_{s=1}^m\lvert\pi_s\rvert < k_i \}$. For all tuples $\vec\pi$ and $\vec\pi'$ we have $\mathsf a(\vec\pi)\sim_0 \mathsf a(\vec\pi')$. If for tuples $\vec\pi$ and $\vec\pi'$ and some $n\geqslant 1$ we have $\pi_1\sim_n\pi'_1,\ldots,\pi_m\sim_n\pi'_m$, then $\sum_{s=1}^m\lvert\pi_s\rvert=\sum_{s=1}^m\lvert\pi_s'\rvert=t$ and $\mathsf a(\vec\pi)=\mathsf a_j(\vec\pi)\sim_n \mathsf a_j(\vec\pi')=\mathsf a(\vec\pi')$, where $j=\min \{ i \in \mathbb{N} \mid n_i= n \text{ and } t < k_i \}$. Therefore $\mathsf a(\vec\pi)\sim_n \mathsf a(\vec\pi')$ and the mapping $\mathsf a$ is non-expansive.





\end{proof}

In an ultrametric space $(M,d)$, a function $f\colon M \to M$ is called \emph{(strictly) contractive} if $d(f(x),f(y)) < d(x,y)$ when $x \neq y$.

Notice that any operator $\mathsf U\colon\mathcal F_m\to \mathcal F_m$ is strictly contractive if and only if for any $\mathsf a,\mathsf b\in\mathcal F_m$, and any $n,k\in\mathbb N$ we have 
$$\mathsf a\sim_{n,k}\mathsf b\Rightarrow \mathsf U(\mathsf a)\sim_{n,k+1}\mathsf U(\mathsf b).$$

Now we state a generalization of the Banach's fixed-point theorem for ultrametric spaces that will be used in the next sections.

\begin{thm}[Prie\ss-Crampe \cite{PrCr}]\label{fixpoint}
Let $(M,d)$ be a non-empty spherically complete ultrametric space. Then every strictly contractive mapping $f\colon M\to M$ has a unique fixed-point. 
\end{thm}

\section{Admissible rules}\label{SecAdm}



In this section, for the system $\mathsf{Grz}_{\infty} +\mathsf{cut}$, we state admissibility of auxiliary inference rules, which will be used in the proof of the  cut-elimination theorem. 

Recall that the set $\mathcal P$ of all $\infty$-proofs of the system $\mathsf{Grz}_{\infty} +\mathsf{cut} $ can be considered as an ultrametric space with the metric $d_{\mathcal P}$.

By $\mathcal{P}_n$ we denote the set of all $\infty$-proofs that do not contain applications of the cut rule in their $n$-fragments. We also set $\mathcal{P}_0= \mathcal{P}$.

A mapping $\mathsf u:\mathcal P^m\to \mathcal P$ is called \emph{adequate} if for any $n\in\mathbb N$ we have $\mathsf u(\pi_1,\ldots,\pi_n)\in\mathcal P_n$, whenever $\pi_i\in\mathcal P_n$ for all $i\leqslant n$.

In $\mathsf{Grz}_{\infty} +\mathsf{cut}$, we call a single-premise inference rule \emph{strongly admissible} if there is a non-expansive adequate mapping $\mathsf{u}\colon\mathcal{P} \to \mathcal{P}$ that maps any $\infty$-proof of the premise of the rule to an $\infty$-proof of the conclusion. The mapping $\mathsf{u}$ must also satisfy one additional condition: $\lvert \mathsf{u}(\pi)\rvert \leqslant \lvert \pi \rvert$ for any $\pi \in \mathcal{P}$.

In the following lemmata, non-expansive mappings are defined in a standard way by induction on the local heights of $\infty$-proofs for the premises. So we omit further details.

\begin{lem}\label{strongweakening}
For any finite multisets of formulas $\Pi$ and $\Sigma$, the inference rule
\begin{gather*}
\AXC{$\Gamma\Rightarrow\Delta$}
\LeftLabel{$\mathsf{wk}_{\Pi, \Sigma}$}
\UIC{$\Pi,\Gamma\Rightarrow\Delta,\Sigma$}
\DisplayProof
\end{gather*}
is strongly admissible in $\mathsf{Grz}_{\infty} +\mathsf{cut}$. 
\end{lem}

\begin{lem}\label{inversion}
For any formulas $A$ and $B$, the rules
\begin{gather*}
\AXC{$\Gamma , A \rightarrow B \Rightarrow  \Delta$}
\LeftLabel{$\mathsf{li}_{A \to B}$}
\UIC{$\Gamma ,B \Rightarrow  \Delta$}
\DisplayProof\qquad
\AXC{$\Gamma , A \rightarrow B \Rightarrow  \Delta$}
\LeftLabel{$\mathsf{ri}_{A \to B}$}
\UIC{$\Gamma  \Rightarrow  A, \Delta$}
\DisplayProof\\\\
\AXC{$\Gamma  \Rightarrow   A \rightarrow B, \Delta$}
\LeftLabel{$\mathsf{i}_{A \to B}$}
\UIC{$\Gamma ,A \Rightarrow B, \Delta$}
\DisplayProof\qquad
\AXC{$\Gamma  \Rightarrow   \bot, \Delta$}
\LeftLabel{$\mathsf{i}_{\bot}$}
\UIC{$\Gamma  \Rightarrow \Delta$}
\DisplayProof\qquad 
\AXC{$\Gamma \Rightarrow \Box A, \Delta$}
\LeftLabel{$\mathsf{li}_{\:\Box A}$}
\UIC{$\Gamma \Rightarrow A, \Delta$}
\DisplayProof
\end{gather*}
are strongly admissible in $\mathsf{Grz}_{\infty} +\mathsf{cut}$.
\end{lem}

\begin{lem}\label{weakcontraction}
For any atomic proposition $p$, the rules
\begin{gather*}
\AXC{$\Gamma , p,p \Rightarrow  \Delta$}
\LeftLabel{$\mathsf{acl}_{p}$}
\UIC{$\Gamma ,p \Rightarrow  \Delta$}
\DisplayProof\qquad
\AXC{$\Gamma \Rightarrow p,p, \Delta$}
\LeftLabel{$\mathsf{acr}_{p}$}
\UIC{$\Gamma \Rightarrow p, \Delta$}
\DisplayProof
\end{gather*}
are strongly admissible in $\mathsf{Grz}_{\infty} +\mathsf{cut}$.
\end{lem}

\section{Cut elimination}
\label{SecCut}


In this section, we construct a continuous function from $\mathcal{P}$ to $\mathcal{P}$, which maps any $\infty$-proof of the system $\mathsf{Grz}_{\infty} +\mathsf{cut}$ to a cut-free $\infty$-proof of the same sequent. 

Let us call a pair of $\infty$-proofs $(\pi,\tau)$ a \emph{cut pair} if $\pi$ is an $\infty$-proof of the sequent $\Gamma\Rightarrow\Delta,A$ and $\tau$ is an $\infty$-proof of the sequent $A, \Gamma\Rightarrow \Delta$ for some $\Gamma, \Delta$, and $A$. For a cut pair $(\pi,\tau)$, we call the sequent $\Gamma\Rightarrow \Delta$ its \emph{cut result} and the formula $A$ its cut formula.

For a modal formula $A$, a non-expansive mapping $\mathsf{u}$ from $\mathcal P \times \mathcal P$ to $\mathcal P$ is called \emph{$A$-removing} if it maps every cut pair $(\pi,\tau)$ with the cut formula $A$ to an $\infty$-proof of its cut result.
By $\mathcal R_A$, let us denote the set of all $A$-removing mappings.

\begin{lem}\label{Comp R_A}
For each formula $A$, the pair $(\mathcal R_A,l_2)$ is a non-empty spherically complete ultrametric space.

\end{lem}
\begin{proof}
The proof of spherical completeness of the space $(\mathcal R_A,l_2)$ is analoguos to the proof of Proposition \ref{SphCom}.

We only need to check that the set $\mathcal R_A$ is non-empty. Consider the mapping $\mathsf u_{cut}:\mathcal P^2\to\mathcal P$ that is defined as follows. For a cut pair $(\pi,\tau)$ with the cut formula $A$, it joins the $\infty$-proofs $\pi$ and $\tau$ with an appropriate instance of the rule $(\mathsf{cut})$ . For all other pairs, the mapping $\mathsf u_{cut}$ returns the first argument.

Clearly, $\mathsf u_{cut}$ is non-expansive and therefore lies in $\mathcal R_A$.
\end{proof}

In what follows, we use nonexpansive adequate mappings $ \mathsf{wk}_{\Pi, \Sigma}$, $\mathsf{li}_{A\to B}$, $\mathsf{ri}_{A\to B}$, $\mathsf{i}_{A\to B}$, $\mathsf{i}_\bot$, $\mathsf{li}_{\Box A}$, $\mathsf{acl}_{p}$, $\mathsf{acr}_{p}$ from Lemma \ref{strongweakening}, Lemma \ref{inversion}, and Lemma \ref{weakcontraction}. 

\begin{lem}\label{repadeq}
For any atomic proposition $p$, there exists an adequate $p$-removing mapping $\mathsf {re}_p$.
\end{lem}
\begin{proof}

Assume we have two $\infty$-proofs $\pi$ and $\tau$. If the pair $(\pi,\tau)$ is not a cut pair or is a cut pair with the cut formula being not $p$, then we put $\mathsf{re}_{p}(\pi,\tau)=\pi$. 
Otherwise, we define $\mathsf{re}_{p}(\pi,\tau)$ by induction on $\lvert \pi\rvert$. Let the cut result of  the pair $(\pi,\tau)$ be $\Gamma\Rightarrow \Delta$.

If $\lvert \pi\rvert=0$, then $\Gamma\Rightarrow \Delta, p$ is an initial sequent. Suppose that $\Gamma\Rightarrow \Delta$ is also an initial sequent. Then $\mathsf{re}_{p}(\pi,\tau)$ is defined as the $\infty$-proof consisting only of the sequent $\Gamma\Rightarrow \Delta$. If $\Gamma\Rightarrow \Delta$ is not an initial sequent, then $\Gamma$ has the form $p,\Phi$, and $\tau$ is an $\infty$-proof of the sequent $p,p,\Phi \Rightarrow \Delta$. Applying the non-expansive adequate mapping $\mathsf{acl}_p$ from Lemma \ref{weakcontraction}, we put $\mathsf{re}_{p}(\pi,\tau) := \mathsf{acl}_p (\tau)$.   

Now suppose that $\lvert \pi \rvert >0$. We define $\mathsf{re}_{p}(\pi,\tau)$ according to the last application of an inference rule in $\pi$:  
\begin{align*}
\left(
\AXC{$\pi_0$}
\noLine
\UIC{$\Gamma,B\Rightarrow C,\Sigma,p$}
\LeftLabel{$\mathsf{\to_R}$}
\UIC{$\Gamma\Rightarrow B\to C,\Sigma,p$}
\DisplayProof
,\tau
\right)
&\longmapsto
\AXC{$\mathsf{re}_p(\pi_0, \mathsf{i}_{B \to C}(\tau))$}
\noLine
\UIC{$\Gamma,B\Rightarrow C,\Sigma$}
\LeftLabel{$\mathsf{\to_R}$}
\RightLabel{ ,}
\UIC{$\Gamma\Rightarrow B\to C,\Sigma$}
\DisplayProof
\\\\
\left(
\AXC{$\pi_0$}
\noLine
\UIC{$\Sigma, C\Rightarrow \Delta, p$}
\AXC{$\pi_1$}
\noLine
\UIC{$\Sigma \Rightarrow B,\Delta, p$}
\LeftLabel{$\mathsf{\to_L}$}
\BIC{$\Sigma, B\to C\Rightarrow \Delta, p$}
\DisplayProof
,\tau
\right)
&\longmapsto
\AXC{$\mathsf{re}_p (\pi_0, \mathsf{li}_{B\to C} (\tau))$}
\noLine
\UIC{$\Sigma, C\Rightarrow \Delta$}
\AXC{$\mathsf{re}_p (\pi_1, \mathsf{ri}_{B\to C} (\tau))$}
\noLine
\UIC{$\Sigma \Rightarrow B,\Delta$}
\LeftLabel{$\mathsf{\to_L}$}
\RightLabel{ ,}
\BIC{$\Sigma, B\to C\Rightarrow \Delta$}
\DisplayProof
\\\\
\left(
\AXC{$\pi_0$}
\noLine
\UIC{$\Sigma,B,\Box B\Rightarrow \Delta,p$}
\LeftLabel{$\mathsf{refl}$}
\UIC{$\Sigma,\Box B\Rightarrow \Delta,p$}
\DisplayProof
,\tau
\right)
&\longmapsto
\AXC{$\mathsf{re}_p(\pi_0, \mathsf{wk}_{ B, \emptyset} (\tau)$}
\noLine
\UIC{$\Sigma,B,\Box B\Rightarrow \Delta$}
\LeftLabel{$\mathsf{refl}$}
\RightLabel{ ,}
\UIC{$\Sigma,\Box B\Rightarrow \Delta$}
\DisplayProof
\\\\
\left(
\AXC{$\pi_0$}
\noLine
\UIC{$\Gamma\Rightarrow B,\Delta, p$}
\AXC{$\pi_1$}
\noLine
\UIC{$\Gamma,B \Rightarrow \Delta, p$}
\LeftLabel{$\mathsf{cut}$}
\BIC{$\Gamma\Rightarrow \Delta, p$}
\DisplayProof
,\tau
\right)
&\longmapsto
\AXC{$\mathsf{re}_p (\pi_0, \mathsf{wk}_{\emptyset,B} (\tau))$}
\noLine
\UIC{$\Gamma\Rightarrow B,\Delta$}
\AXC{$\mathsf{re}_p (\pi_1, \mathsf{wk}_{B,\emptyset} (\tau))$}
\noLine
\UIC{$\Gamma, B\Rightarrow \Delta$}
\LeftLabel{$\mathsf{cut}$}
\RightLabel{ ,}
\BIC{$\Gamma\Rightarrow \Delta$}
\DisplayProof
\\\\
\left(
\AXC{$\pi_0$}
\noLine
\UIC{$\Phi, \Box \Pi \Rightarrow B, \Sigma, p$}
\AXC{$\pi_1$}
\noLine
\UIC{$\Box \Pi \Rightarrow B$}
\LeftLabel{$\mathsf{\Box}$}
\BIC{$\Phi, \Box \Pi \Rightarrow \Box B, \Sigma, p$}
\DisplayProof
,\tau\right)
&\longmapsto
\AXC{$\mathsf{re}_p(\pi_0, \mathsf{li}_{\: \Box B}(\tau))$}
\noLine
\UIC{$\Phi, \Box \Pi \Rightarrow B, \Sigma$}
\AXC{$\pi_1$}
\noLine
\UIC{$\Box \Pi \Rightarrow B$}
\LeftLabel{$\mathsf{\Box}$}
\RightLabel{ .}
\BIC{$\Phi, \Box \Pi \Rightarrow \Box B, \Sigma$}
\DisplayProof
\end{align*}

The mapping $\mathsf{re}_p$ is well defined. 

Now we claim that $\mathsf{re}_p$ is non-expansive. It sufficient to check that for any pairs $(\pi,\tau)$ and $(\pi^\prime,\tau^\prime)$, and any $n\in\mathbb N$ we have 
\[\pi\sim_n\pi',\; \tau\sim_n\tau'\Rightarrow  \mathsf{re}_p(\pi,\tau)\sim_n \mathsf{re}_p(\pi^\prime,\tau^\prime).\]

Assume we have two pairs of $\infty$-proofs $(\pi,\tau)$ and $(\pi^\prime,\tau^\prime)$ such that $\pi\sim_n\pi'$, $\tau\sim_n\tau'$.
If $n=0$, then trivially $ \mathsf{re}_p(\pi,\tau)\sim_n \mathsf{re}_p(\pi^\prime,\tau^\prime)$. If $n>0$, then main fragments of $\pi$ and $\pi^\prime$ ($\tau$ and $\tau^\prime$) are identical.
Suppose that $(\pi,\tau)$ is not a cut pair or is a cut pair with the cut formula being not $p$. Then the same condition holds for the pair $(\pi^\prime,\tau^\prime)$. In this case, we have 
\[ \mathsf{re}_p(\pi,\tau) = \pi \sim_n \pi^\prime =  \mathsf{re}_p(\pi^\prime,\tau^\prime).\] 

Otherwise, $(\pi,\tau)$ and $(\pi^\prime,\tau^\prime)$ 
are cut pairs with the cut formula $p$.
Now the statement $\mathsf{re}_p(\pi,\tau)\sim_n \mathsf{re}_p(\pi^\prime,\tau^\prime)$ is proved by induction on $\lvert \pi\rvert=  \lvert \pi^\prime\rvert$ with the case analysis according to the definition of $\mathsf{re}_p$. 

If the last inference in $\pi$ is an application of the rule $(\mathsf{\to_R})$, then $\pi$ and $\pi^\prime$ have the following forms
\[\AXC{$\pi_0$}
\noLine
\UIC{$\Gamma,B\Rightarrow C,\Sigma,p$}
\LeftLabel{$\mathsf{\to_R}$}
\UIC{$\Gamma\Rightarrow B\to C,\Sigma,p$}
\DisplayProof\qquad
\AXC{$\pi^\prime_0$}
\noLine
\UIC{$\Gamma,B\Rightarrow C,\Sigma,p$}
\LeftLabel{$\mathsf{\to_R}$}
\RightLabel{ ,}
\UIC{$\Gamma\Rightarrow B\to C,\Sigma,p$}
\DisplayProof
\]
where $\pi_0 \sim_n \pi^\prime_0$. Since $\mathsf{i}_{B \to C}$ is non-expansive, we have $\mathsf{i}_{B \to C}(\tau) \sim_n \mathsf{i}_{B \to C}(\tau^\prime)$. Consequently, by the induction hypothesis for pairs $(\pi_0, \mathsf{i}_{B \to C}(\tau))$ and $(\pi^\prime_0, \mathsf{i}_{B \to C}(\tau^\prime))$, we have $\mathsf{re}_p(\pi_0, \mathsf{i}_{B \to C}(\tau)) \sim_n \mathsf{re}_p(\pi^\prime_0, \mathsf{i}_{B \to C}(\tau))$. Thus we obtain $\mathsf{re}_p(\pi,\tau)\sim_n \mathsf{re}_p(\pi^\prime,\tau^\prime)$.

Consider the case when the last inference in $\pi$ is an application of the rule $(\mathsf{\Box})$. We have that $\pi$ and $\pi^\prime$ have the following forms
\[\AXC{$\pi_0$}
\noLine
\UIC{$\Phi, \Box \Pi \Rightarrow B, \Sigma, p$}
\AXC{$\pi_1$}
\noLine
\UIC{$\Box \Pi \Rightarrow B$}
\LeftLabel{$\mathsf{\Box}$}
\BIC{$\Phi, \Box \Pi \Rightarrow \Box B, \Sigma, p$}
\DisplayProof\qquad
\AXC{$\pi^\prime_0$}
\noLine
\UIC{$\Phi, \Box \Pi \Rightarrow B, \Sigma, p$}
\AXC{$\pi^\prime_1$}
\noLine
\UIC{$\Box \Pi \Rightarrow B$}
\LeftLabel{$\mathsf{\Box}$}
\RightLabel{ ,}
\BIC{$\Phi, \Box \Pi \Rightarrow \Box B, \Sigma, p$}
\DisplayProof
\] 
where $\pi_0 \sim_n \pi^\prime_0$ and $\pi_1 \sim_{n-1} \pi^\prime_1$. Since $\mathsf{li}_{\: \Box B}$ is non-expansive, we have $\mathsf{li}_{\: \Box B}(\tau) \sim_n \mathsf{li}_{\: \Box B}(\tau^\prime)$. Consequently, by the induction hypothesis for pairs $(\pi_0, \mathsf{li}_{\: \Box B}(\tau))$ and $(\pi^\prime_0, \mathsf{li}_{\: \Box B}(\tau^\prime))$, we have $\mathsf{re}_p(\pi_0, \mathsf{li}_{\: \Box B}(\tau)) \sim_n \mathsf{re}_p(\pi^\prime_0, \mathsf{li}_{\: \Box B}(\tau))$. Thus we obtain $\mathsf{re}_p(\pi,\tau)\sim_n \mathsf{re}_p(\pi^\prime,\tau^\prime)$.

All the other cases are treated similarly, so we omit them. We have that the mapping $\mathsf{re}_p$ is non-expansive.
It remains to check that the mapping $\mathsf{re}_p$ is adequate.
 
Assume we have a pair of $\infty$-proofs $(\pi,\tau)$ such that $\pi \in \mathcal{P}_n$ and $\tau\in\mathcal{P}_n $. We claim that $\mathsf{re}_p(\pi,\tau) \in \mathcal{P}_n$. If $n=0$, then trivially $ \mathsf{re}_p(\pi,\tau) \in \mathcal{P}_0=\mathcal{P}$. If the pair $(\pi,\tau)$ is not a cut pair or is a cut pair with the cut formula being not $p$, then we have $ \mathsf{re}_p(\pi,\tau) = \pi \in \mathcal{P}_n$. 

Now suppose $(\pi,\tau)$ is a cut pair with the cut formula $p$ and $n>0$. We proceed by induction on $\lvert \pi\rvert$ with the case analysis according to the definition of $\mathsf{re}_p$. Let us consider the cases of inference rules $(\mathsf{\to_L})$ and $(\mathsf{\Box})$. 

If the last inference in $\pi$ is an application of the rule $(\mathsf{\to_L})$, then $\pi$ has the following form
\[
\AXC{$\pi_0$}
\noLine
\UIC{$\Sigma, C\Rightarrow \Delta, p$}
\AXC{$\pi_1$}
\noLine
\UIC{$\Sigma \Rightarrow B,\Delta, p$}
\LeftLabel{$\mathsf{\to_L}$}
\BIC{$\Sigma, B\to C\Rightarrow \Delta, p$}
\DisplayProof
\]
where $\pi_0, \pi_1 \in \mathcal{P}_n$. Since $\mathsf{li}_{B \to C}$ and $\mathsf{ri}_{B \to C}$ are  adequate, we have $\mathsf{li}_{B \to C}(\tau)\in \mathcal{P}_n$ and $\mathsf{ri}_{B \to C}(\tau) \in \mathcal{P}_n$. Consequently, by the induction hypothesis for the pairs $(\pi_0, \mathsf{li}_{B \to C}(\tau))$ and $(\pi_1, \mathsf{ri}_{B \to C}(\tau))$, we have $\mathsf{re}_p(\pi_0, \mathsf{li}_{B \to C}(\tau))\in \mathcal{P}_n$ and $\mathsf{re}_p(\pi_1, \mathsf{ri}_{B \to C}(\tau))\in \mathcal{P}_n$. Thus we obtain $\mathsf{re}_p(\pi,\tau)\in \mathcal{P}_n$.

Consider the case when the last inference in $\pi$ is an application of the rule $(\mathsf{\Box})$.
Then $\pi$ has the form
\[\AXC{$\pi_0$}
\noLine
\UIC{$\Phi, \Box \Pi \Rightarrow B, \Sigma, p$}
\AXC{$\pi_1$}
\noLine
\UIC{$\Box \Pi \Rightarrow B$}
\LeftLabel{$\mathsf{\Box}$}
\RightLabel{ ,}
\BIC{$\Phi, \Box \Pi \Rightarrow \Box B, \Sigma, p$}
\DisplayProof
\] 
where $\pi_0 \in \mathcal{P}_n$ and $\pi_1 \in \mathcal{P}_{n-1}$. Since $\mathsf{li}_{\: \Box B}$ is adequate, we have $\mathsf{li}_{\: \Box B}(\tau) \in \mathcal{P}_n$. Consequently, by the induction hypothesis for the pair $(\pi_0, \mathsf{li}_{\: \Box B}(\tau))$, we have $\mathsf{re}_p(\pi_0, \mathsf{li}_{\: \Box B}(\tau)) \in \mathcal{P}_n$. Thus we obtain $\mathsf{re}_p(\pi,\tau)\in \mathcal{P}_n$.

Notice that the last inference in $\pi$ differs from an application of the rule $(\mathsf{cut})$. The remaining cases of inference rules $(\to_\mathsf{R})$ and $(\mathsf{refl})$ are treated in a similar way to the case of $(\to_\mathsf{L})$, so we omit them.

\end{proof}

\begin{lem}\label{reboxadeq}
Given an adequate $B$-removing mapping $\mathsf{re}_B$, there exists an adequate $\Box B$-removing mapping $\mathsf{re}_{\Box B}$.
\end{lem}
\begin{proof}
Assume we have an adequate $B$-removing mapping $\mathsf{re}_B$. 
The required $\Box B$-removing mapping $\mathsf{re}_{\Box B}$ is obtained as the fixed-point of a contractive operator $\mathsf G_{\Box B} \colon \mathcal R_{\Box B} \to \mathcal R_{\Box B}$. 

For a mapping $\mathsf u\in \mathcal R_{\Box B}$ and a pair of $\infty$-proofs $(\pi,\tau)$, the $\infty$-proof $\mathsf G_{\Box B}(\mathsf u)(\pi,\tau)$ is defined as follows. 
If $(\pi,\tau)$ is not a cut pair or a cut pair with 
the cut formula being not $\Box B$, then we put $\mathsf G_{\Box B}(\mathsf u)(\pi,\tau)=\pi$.



Now let $(\pi,\tau)$ be a cut pair with 
the cut formula $\Box B$ and the cut result $\Gamma\Rightarrow \Delta$.
If $\lvert \pi\rvert=0$ or $\lvert \tau \rvert=0$, then $\Gamma\Rightarrow \Delta$ is an initial sequent. In this case, we define $\mathsf G_{\Box B}(\mathsf u)(\pi,\tau)$ as the $\infty$-proof consisting only of the sequent $\Gamma\Rightarrow \Delta$. 

Suppose that $\lvert \pi\rvert>0$ and $\lvert \tau \rvert>0$. We define $\mathsf G_{\Box B}(\mathsf u)(\pi,\tau)$ according to the last application of an inference rule in $\pi$:\\
\begin{align*}
\left(
\AXC{$\pi_0$}
\noLine
\UIC{$\Gamma,C\Rightarrow D,\Sigma,\Box B$}
\LeftLabel{$\mathsf{\to_R}$}
\UIC{$\Gamma\Rightarrow C\to D,\Sigma,\Box B$}
\DisplayProof
,\tau
\right)
&\longmapsto
\AXC{$\mathsf{u}(\pi_0, \mathsf{i}_{C \to D}(\tau))$}
\noLine
\UIC{$\Gamma,C\Rightarrow D,\Sigma$}
\LeftLabel{$\mathsf{\to_R}$}
\RightLabel{ ,}
\UIC{$\Gamma\Rightarrow C\to D,\Sigma$}
\DisplayProof
\\\\
\left(
\AXC{$\pi_0$}
\noLine
\UIC{$\Sigma, D\Rightarrow \Delta, \Box B$}
\AXC{$\pi_1$}
\noLine
\UIC{$\Sigma \Rightarrow C,\Delta, \Box B$}
\LeftLabel{$\mathsf{\to_L}$}
\BIC{$\Sigma, C\to D\Rightarrow \Delta, \Box B$}
\DisplayProof
,\tau
\right)
&\longmapsto
\AXC{$\mathsf{u} (\pi_0, \mathsf{li}_{C\to D} (\tau))$}
\noLine
\UIC{$\Sigma, D\Rightarrow \Delta$}
\AXC{$\mathsf{u} (\pi_1, \mathsf{ri}_{C\to D} (\tau))$}
\noLine
\UIC{$\Sigma \Rightarrow C,\Delta$}
\LeftLabel{$\mathsf{\to_L}$}
\RightLabel{ ,}
\BIC{$\Sigma, C\to D\Rightarrow \Delta$}
\DisplayProof
\end{align*}
\begin{align*}
\left(
\AXC{$\pi_0$}
\noLine
\UIC{$\Sigma,C,\Box C\Rightarrow \Delta,\Box B$}
\LeftLabel{$\mathsf{refl}$}
\UIC{$\Sigma,\Box C\Rightarrow \Delta,\Box B$}
\DisplayProof
,\tau
\right)
&\longmapsto
\AXC{$\mathsf{u}(\pi_0, \mathsf{wk}_{ C, \emptyset} (\tau)$}
\noLine
\UIC{$\Sigma,C,\Box C\Rightarrow \Delta$}
\LeftLabel{$\mathsf{refl}$}
\RightLabel{ ,}
\UIC{$\Sigma,\Box C\Rightarrow \Delta$}
\DisplayProof
\\\\
\left(
\AXC{$\pi_0$}
\noLine
\UIC{$\Gamma\Rightarrow C,\Delta, \Box B$}
\AXC{$\pi_1$}
\noLine
\UIC{$\Gamma,C \Rightarrow \Delta, \Box B$}
\LeftLabel{$\mathsf{cut}$}
\BIC{$\Gamma\Rightarrow \Delta, \Box B$}
\DisplayProof
,\tau
\right)
&\longmapsto
\AXC{$\mathsf{u} (\pi_0, \mathsf{wk}_{\emptyset,C} (\tau))$}
\noLine
\UIC{$\Gamma\Rightarrow C,\Delta$}
\AXC{$\mathsf{u} (\pi_1, \mathsf{wk}_{C,\emptyset} (\tau))$}
\noLine
\UIC{$\Gamma, C\Rightarrow \Delta$}
\LeftLabel{$\mathsf{cut}$}
\RightLabel{ ,}
\BIC{$\Gamma\Rightarrow \Delta$}
\DisplayProof
\\\\
\left(
\AXC{$\pi_0$}
\noLine
\UIC{$\Phi, \Box \Pi \Rightarrow C, \Sigma, \Box B$}
\AXC{$\pi_1$}
\noLine
\UIC{$\Box \Pi \Rightarrow C$}
\LeftLabel{$\mathsf{\Box}$}
\BIC{$\Phi, \Box \Pi \Rightarrow \Box C, \Sigma, \Box B$}
\DisplayProof
,\tau\right)
&\longmapsto
\AXC{$\mathsf{u}(\pi_0, \mathsf{li}_{\: \Box C}(\tau))$}
\noLine
\UIC{$\Phi, \Box \Pi \Rightarrow C, \Sigma$}
\AXC{$\pi_1$}
\noLine
\UIC{$\Box \Pi \Rightarrow C$}
\LeftLabel{$\mathsf{\Box}$} 
\RightLabel{ .}
\BIC{$\Phi, \Box \Pi \Rightarrow \Box C, \Sigma$}
\DisplayProof
\end{align*}

Consider the case when $\pi$ has the form 
\[
\AXC{$\pi_0$}
\noLine
\UIC{$\Phi, \Box \Pi \Rightarrow B, \Sigma$}
\AXC{$\pi_1$}
\noLine
\UIC{$\Box \Pi \Rightarrow B$}
\LeftLabel{$\mathsf{\Box}$}
\RightLabel{ .}
\BIC{$\Phi, \Box \Pi \Rightarrow \Box B, \Sigma$}
\DisplayProof
\]
We define $\mathsf G_{\Box B}(\mathsf u)(\pi,\tau)$ according to the last application of an inference rule in $\tau$. 
If the last inference is an application of the rule $\mathsf{(refl)}$ with the principal formula being not $\Box B$ or is an application of the rule $\mathsf{(\Box)}$ without the formula $\Box B$ in the right premise, then $\mathsf G_{\Box B}(\mathsf u)(\pi,\tau)$ is defined similarly to the previous cases of $\mathsf{(refl)}$ and $(\mathsf{\Box})$. Cases of inference rules $\mathsf{(\to_L)}$, $\mathsf{(\to_R)}$, and $\mathsf{(cut)}$ are also completely similar to the previous cases of $\mathsf{(\to_L)}$, $\mathsf{(\to_R)}$, and $\mathsf{(cut)}$.



If the last application of an inference rule in $\tau$ is an application of the rule $\mathsf{(refl)}$ with the principal formula $\Box B$, then we put
\[\mathsf G_{\Box B}(\mathsf u) \colon
\left(\pi,
\AXC{$\tau_0$}
\noLine
\UIC{$\Gamma,B,\Box B\Rightarrow \Delta$}
\LeftLabel{$\mathsf{refl}$}
\UnaryInfC{$\Gamma,\Box B\Rightarrow \Delta$}
\DisplayProof
\right)
\longmapsto\mathsf{re}_{B}(\pi_0,\mathsf{u}(\mathsf{wk}_{B,\emptyset}(\pi),\tau_0)).
\]

It remains to consider the case when $\tau$ has the form
\[
\AxiomC{$\tau_0$}
\noLine
\UnaryInfC{$\Phi^\prime, \Box B, \Box \Pi^\prime \Rightarrow C, \Sigma^\prime$}
\AxiomC{$\tau_1$}
\noLine
\UnaryInfC{$\Box B,\Box \Pi^\prime \Rightarrow C$}
\LeftLabel{$\mathsf{\Box}$}
\RightLabel{.}
\BinaryInfC{$\Phi^\prime, \Box B, \Box \Pi^\prime \Rightarrow \Box C, \Sigma^\prime$}
\DisplayProof
\]
We define $\mathsf G_{\Box B}(\mathsf u)(\pi,
\tau)$ as
\[
\AxiomC{$\mathsf{u}(\mathsf{li}_{\Box C}(\pi),\tau_0)$}
\noLine
\UnaryInfC{$\Phi^\prime, \Box \Pi^\prime \Rightarrow C, \Sigma^\prime$}
\AxiomC{$
\mathsf{u}
\left(\pi',
\mathsf{wk}_{\Box\Pi\backslash\Box\Pi^\prime,\varnothing}(\tau_1)
\right)$}
\noLine
\UnaryInfC{$\Box\Pi\backslash \Box\Pi^\prime, \Box\Pi^\prime\Rightarrow C$}
\LeftLabel{$\mathsf{\Box}$}
\RightLabel{ ,}
\BinaryInfC{$\Phi^\prime, \Box \Pi^\prime \Rightarrow \Box C, \Sigma^\prime$}
\DisplayProof
\]
where $\pi'$ equals to
\[
\AxiomC{$\mathsf{wk}_{\Box\Pi^\prime\backslash\Box\Pi,C}(\pi_1)$}
\noLine
\UnaryInfC{$\Box\Pi^\prime\backslash \Box\Pi, \Box\Pi\Rightarrow B,C$}
\AxiomC{$\pi_1$}
\noLine
\UnaryInfC{$\Box\Pi\Rightarrow B$}
\LeftLabel{$\mathsf{\Box}$}
\RightLabel{ .}
\BinaryInfC{$\Box\Pi^\prime\backslash \Box\Pi,\Box\Pi\Rightarrow\Box B, C$}
\DisplayProof
\]
Notice that $\Box\Pi^\prime\backslash \Box\Pi,\Box\Pi=\Box\Pi\backslash \Box\Pi^\prime,\Box\Pi^\prime$.

Now the operator $\mathsf {G_{\Box B}}$ is well-defined.
By the case analysis according to the definition of $\mathsf {G_{\Box B}}$, we see that $\mathsf {G_{\Box B}} (\mathsf u) $ is non-expansive and belongs to $\mathcal R_{\Box B}$ whenever $\mathsf u \in \mathcal R_{\Box B}$.

We claim that $\mathsf {G_{\Box B}}$ is contractive. It sufficient to check that for any $\mathsf u, \mathsf v\in \mathcal R_{\Box B}$ and any $n,k\in\mathbb N$ we have 
\[\mathsf u\sim_{n,k}\mathsf v \Rightarrow \mathsf {G_{\Box B}(u)}\sim_{n,k+1}\mathsf{G_{\Box B}(v)}.\]


Assume there are two $\Box B$-removing mappings $\mathsf u$ and $\mathsf v$ such that $\mathsf u\sim_{n,k}\mathsf v$. Consider an arbitrary pair of $\infty$-proofs $(\pi,\tau)$. By the case analysis according to the definition of $\mathsf {G_{\Box B}}$, we prove that $\mathsf G_{\Box B}(\mathsf u)(\pi,\tau)\sim_{n}\mathsf G_{\Box B}(\mathsf v)(\pi,\tau)$. In addition, we check that if $\lvert\pi\rvert+\lvert\tau\lvert<k+1$, then $\mathsf G_{\Box B}(\mathsf u)(\pi,\tau)\sim_{n+1}\mathsf G_{\Box B}(\mathsf v)(\pi,\tau)$. 

If the pair $(\pi,\tau)$ is not a cut pair or is a cut pair with the cut formula being not $\Box B$, then we have $ \mathsf G_{\Box B}(\mathsf u)(\pi,\tau) = \pi =\mathsf G_{\Box B}(\mathsf v)(\pi,\tau)$. 
Otherwise, consider last inferences in $\pi$ and $\tau$.  

Suppose that the $\infty$-proof $\pi$ has the form
\[\AXC{$\pi_0$}
\noLine
\UIC{$\Sigma,C,\Box C\Rightarrow \Delta,\Box B$}
\LeftLabel{$\mathsf{refl}$}
\RightLabel{ .}
\UIC{$\Sigma,\Box C\Rightarrow \Delta,\Box B$}
\DisplayProof
\]
Then we have $\mathsf{u}(\pi_0, \mathsf{wk}_{ C, \emptyset} (\tau))\sim_n \mathsf{v}(\pi_0, \mathsf{wk}_{ C, \emptyset} (\tau))$. Hence, $\mathsf G_{\Box B}(\mathsf u)(\pi,\tau)\sim_{n}\mathsf G_{\Box B}(\mathsf v)(\pi,\tau)$. 

In addition, if $\lvert\pi\rvert+\lvert\tau\lvert<k+1$, then $\lvert\pi_0\rvert+\lvert\mathsf{wk}_{ C, \emptyset} (\tau)\lvert<k$. Thus, $\mathsf{u}(\pi_0, \mathsf{wk}_{ C, \emptyset} (\tau))\sim_{n+1} \mathsf{v}(\pi_0, \mathsf{wk}_{ C, \emptyset} (\tau))$. We obtain that $\mathsf G_{\Box B}(\mathsf u)(\pi,\tau)\sim_{n+1}\mathsf G_{\Box B}(\mathsf v)(\pi,\tau)$. 

Now let us consider the case when $\pi$ and $\tau$ have the forms
\[
\AXC{$\pi_0$}
\noLine
\UIC{$\Phi, \Box \Pi \Rightarrow B, \Sigma$}
\AXC{$\pi_1$}
\noLine
\UIC{$\Box \Pi \Rightarrow B$}
\LeftLabel{$\mathsf{\Box}$}
\RightLabel{ ,}
\BIC{$\Phi, \Box \Pi \Rightarrow \Box B, \Sigma$}
\DisplayProof\quad
\AXC{$\tau_0$}
\noLine
\UIC{$\Gamma,B,\Box B\Rightarrow \Delta$}
\LeftLabel{$\mathsf{refl}$}
\RightLabel{ .}
\UnaryInfC{$\Gamma,\Box B\Rightarrow \Delta$}
\DisplayProof
\]
We see that $\mathsf{u}(\mathsf{wk}_{B,\emptyset}(\pi),\tau_0) \sim_n \mathsf{v}(\mathsf{wk}_{B,\emptyset}(\pi),\tau_0)$. Since the mapping $\mathsf{re}_{B}$ is non-expansive, we have 
\[\mathsf G_{\Box B}(\mathsf u)(\pi,\tau) = \mathsf{re}_{B}(\pi_0,\mathsf{u}(\mathsf{wk}_{B,\emptyset}(\pi),\tau_0)) \sim_n \mathsf{re}_{B}(\pi_0,\mathsf{v}(\mathsf{wk}_{B,\emptyset}(\pi),\tau_0)) = \mathsf G_{\Box B}(\mathsf v)(\pi,\tau).\]

In addition, if $\lvert\pi\rvert+\lvert\tau\lvert<k+1$, then $\lvert\mathsf{wk}_{B,\emptyset}(\pi)\rvert+\lvert\tau_0\lvert<k$. Consequently, $\mathsf{u}(\mathsf{wk}_{B,\emptyset}(\pi),\tau_0) \sim_{n+1} \mathsf{v}(\mathsf{wk}_{B,\emptyset}(\pi),\tau_0)$.
Thus, we have
\[\mathsf G_{\Box B}(\mathsf u)(\pi,\tau) = \mathsf{re}_{B}(\pi_0,\mathsf{u}(\mathsf{wk}_{B,\emptyset}(\pi),\tau_0)) \sim_{n+1} \mathsf{re}_{B}(\pi_0,\mathsf{v}(\mathsf{wk}_{B,\emptyset}(\pi),\tau_0)) = \mathsf G_{\Box B}(\mathsf v)(\pi,\tau).\]

Let us consider the main case when $\pi$ and $\tau$ have the forms 
\[
\AXC{$\pi_0$}
\noLine
\UIC{$\Phi, \Box \Pi \Rightarrow B, \Sigma$}
\AXC{$\pi_1$}
\noLine
\UIC{$\Box \Pi \Rightarrow B$}
\LeftLabel{$\mathsf{\Box}$}
\RightLabel{ ,}
\BIC{$\Phi, \Box \Pi \Rightarrow \Box B, \Sigma$}
\DisplayProof\quad
\AxiomC{$\tau_0$}
\noLine
\UnaryInfC{$\Phi^\prime, \Box B, \Box \Pi^\prime \Rightarrow C, \Sigma^\prime$}
\AxiomC{$\tau_1$}
\noLine
\UnaryInfC{$\Box B,\Box \Pi^\prime \Rightarrow C$}
\LeftLabel{$\mathsf{\Box}$}
\RightLabel{ .}
\BinaryInfC{$\Phi^\prime, \Box B, \Box \Pi^\prime \Rightarrow \Box C, \Sigma^\prime$}
\DisplayProof
\]
We have $\mathsf{u}\left(\pi',\mathsf{wk}_{\Box\Pi\backslash\Box\Pi^\prime,\varnothing}(\tau_1)\right)\sim_n\mathsf{v}\left(\pi',\mathsf{wk}_{\Box\Pi\backslash\Box\Pi^\prime,\varnothing}(\tau_1)\right)$, and $\mathsf{u}(\mathsf{li}_{\Box C}(\pi),\tau_0)\sim_{n}\mathsf{v}(\mathsf{li}_{\Box C}(\pi),\tau_0)$.
Hence, $\mathsf G_{\Box B}(\mathsf u)(\pi,\tau)\sim_{n}\mathsf G_{\Box B}(\mathsf v)(\pi,\tau)$. 

If $\lvert\pi\rvert+\lvert\tau\lvert<k+1$, then
$\mathsf{u}\left(\pi',\mathsf{wk}_{\Box\Pi\backslash\Box\Pi^\prime,\varnothing}(\tau_1)\right)\sim_n\mathsf{v}\left(\pi',\mathsf{wk}_{\Box\Pi\backslash\Box\Pi^\prime,\varnothing}(\tau_1)\right)$ and $\mathsf{u}(\mathsf{li}_{\Box C}(\pi),\tau_0)\sim_{n+1}\mathsf{v}(\mathsf{li}_{\Box C}(\pi),\tau_0)$, because $\lvert \mathsf{li}_{\Box C}(\pi)\rvert +\lvert \tau_0\rvert< k$. We obtain $\mathsf G_{\Box B}(\mathsf u)(\pi,\tau)\sim_{n+1}\mathsf G_{\Box B}(\mathsf v)(\pi,\tau)$.




All the other cases of lowermost inferences in $\pi$ and $\tau$ are treated similarly, so we omit them. We have that the operator $\mathsf {G_{\Box B}}$ is contractive. 

Now we define the required $\Box B$-removing mapping $\mathsf{re}_{\Box B}$ as the fixed-point of the the operator $\mathsf G_{\Box B} \colon \mathcal R_{\Box B} \to \mathcal R_{\Box B}$, which exists by Lemma \ref{Comp R_A} and Theorem \ref{fixpoint}.


It remains to check that the mapping $\mathsf{re}_{\Box B}$ is adequate. For some $n,k\in\mathbb N$, let us call a mapping $\mathsf u\in \mathcal R_{\Box B}$ $(n,k)$-adequate if it satisfies the following two conditions: $\mathsf u(\pi,\tau)\in\mathcal P_i$ for any $i\leqslant n$ and any $\pi, \tau \in \mathcal P_i$; $\mathsf u(\pi,\tau)\in\mathcal P_{n+1}$ whenever $\pi, \tau \in \mathcal P_{n+1}$ and $\lvert \pi\rvert + \lvert \tau \rvert <k$. 

We claim that $\mathsf G_{\Box B}(\mathsf u)$ is $(n,k+1)$-adequate whenever $\mathsf{u}$ is a $(n,k)$-adequate mapping from $\mathcal R_{\Box B}$. Consider an arbitrary pair of $\infty$-proofs $(\pi,\tau)$ and an $(n,k)$-adequate mapping $\mathsf u\in\mathcal R_{\Box B}$. By the case analysis according to the definition of $\mathsf {G_{\Box B}}$, we prove that $\mathsf G_{\Box B}(\mathsf u)(\pi,\tau)\in \mathcal{P}_i$ if $i\leqslant n$ and $\pi, \tau \in \mathcal P_i$. In addition, we check that $\mathsf G_{\Box B}(\mathsf u)(\pi,\tau)\in \mathcal P_{n+1}$ if $\lvert\pi\rvert+\lvert\tau\lvert<k+1$ and $\pi, \tau \in \mathcal P_{n+1}$. 

If the pair $(\pi,\tau)$ is not a cut pair or is a cut pair with the cut formula being not $\Box B$, then we have $ \mathsf G_{\Box B}(\mathsf u)(\pi,\tau) = \pi $. In this case, the required statements trivially hold.
Otherwise, consider last inferences in $\pi$ and $\tau$.  

Suppose that the $\infty$-proof $\pi$ has the form
\[\AXC{$\pi_0$}
\noLine
\UIC{$\Gamma,C\Rightarrow D,\Sigma,\Box B$}
\LeftLabel{$\mathsf{\to_R}$}
\RightLabel{ .}
\UIC{$\Gamma\Rightarrow C\to D,\Sigma,\Box B$}
\DisplayProof
\]
Since the mapping $\mathsf{i}_{ C\to D}$ is adequate, we have $\mathsf{u}(\pi_0, \mathsf{i}_{ C\to D} (\tau))\in \mathcal{P}_i$ whenever $i\leqslant n$ and $\pi, \tau \in \mathcal P_i$. In this case, we see that $\mathsf G_{\Box B}(\mathsf u)(\pi,\tau)\in  \mathcal P_{i}$.

If $\lvert\pi\rvert+\lvert\tau\lvert<k+1$, then $\lvert\pi_0\rvert+\lvert\mathsf{i}_{ C\to D}(\tau)\lvert<k$. Thus, $\mathsf{u}(\pi_0, \mathsf{i}_{ C\to D} (\tau))\in  \mathcal P_{n+1}$ whenever $\pi, \tau \in \mathcal P_{n+1}$. We obtain that $\mathsf G_{\Box B}(\mathsf u)(\pi,\tau)\in \mathcal P_{n+1}$. 

Suppose that $\pi$ and $\tau$ have the forms 
\[
\AXC{$\pi_0$}
\noLine
\UIC{$\Phi, \Box \Pi \Rightarrow B, \Sigma$}
\AXC{$\pi_1$}
\noLine
\UIC{$\Box \Pi \Rightarrow B$}
\LeftLabel{$\mathsf{\Box}$}
\RightLabel{ ,}
\BIC{$\Phi, \Box \Pi \Rightarrow \Box B, \Sigma$}
\DisplayProof\quad
\AxiomC{$\tau_0$}
\noLine
\UnaryInfC{$\Phi^\prime, \Box B, \Box \Pi^\prime \Rightarrow C, \Sigma^\prime$}
\AxiomC{$\tau_1$}
\noLine
\UnaryInfC{$\Box B,\Box \Pi^\prime \Rightarrow C$}
\LeftLabel{$\mathsf{\Box}$}
\RightLabel{ .}
\BinaryInfC{$\Phi^\prime, \Box B, \Box \Pi^\prime \Rightarrow \Box C, \Sigma^\prime$}
\DisplayProof
\]
Trivially, $\mathsf G_{\Box B}(\mathsf u)(\pi,\tau)\in \mathcal P_{0} = \mathcal{P}$. Thus consider the case when $\pi, \tau \in \mathcal{P}_i$ for some $0<i \leqslant n$. Then we see that $\pi_0, \tau_0 \in \mathcal{P}_i$ and $\pi_1, \tau_1 \in \mathcal{P}_{i-1}$. 
Since the mapping $\mathsf{wk}_{\Box\Pi\backslash\Box\Pi^\prime,\varnothing}$ and $\mathsf{li}_{\Box C}$ are adequate, we have $\mathsf{u}(\mathsf{li}_{\Box C}(\pi),\tau_0)\in \mathcal{P}_i$ and $\mathsf{u}\left(\pi',\mathsf{wk}_{\Box\Pi\backslash\Box\Pi^\prime,\varnothing}(\tau_1)\right)\in \mathcal{P}_{i-1}$.
Hence, $\mathsf G_{\Box B}(\mathsf u)(\pi,\tau)\in \mathcal{P}_i$. 

If $\lvert\pi\rvert+\lvert\tau\lvert<k+1$ and $\pi,\tau \in \mathcal{P}_{n+1}$, then
$\lvert \mathsf{li}_{\Box C}(\pi) \rvert + \lvert \tau_0 \rvert <k$ and $\mathsf{u}(\mathsf{li}_{\Box C}(\pi),\tau_0)\in \mathcal{P}_{n+1}$. Also, we see that $\pi'\in  \mathcal{P}_n$ and $\mathsf{wk}_{\Box\Pi\backslash\Box\Pi^\prime,\varnothing}(\tau_1) \in \mathcal{P}_n$. Hence, $\mathsf{u}\left(\pi',\mathsf{wk}_{\Box\Pi\backslash\Box\Pi^\prime,\varnothing}(\tau_1)\right)\in \mathcal{P}_{n}$.
We obtain $\mathsf G_{\Box B}(\mathsf u)(\pi,\tau)\in \mathcal{P}_{n+1}$.

We omit other cases of lowermost inferences in $\pi$ and $\tau$, because they are treated similarly.


We established that $\mathsf G_{\Box B}(\mathsf u)$ is $(n,k+1)$-adequate for any $(n,k)$-adequate $\mathsf u\in\mathcal R_{\Box B}$. Notice that if a mapping $\mathsf{u} $ is $(n,k)$-adequate for all $k\in\mathbb N$, then it is also $(n+1,0)$-adequate. Now by induction on $n$ with a subinduction on $k$, we immediately obtain that the mapping $\mathsf {re}_{\Box B}$, which is a fixed-point of $\mathsf G_{\Box B}$, is $(n,k)$-adequate for all $n,k\in\mathbb N$. Therefore the mapping $\mathsf {re}_{\Box B}$ is adequate. 


\end{proof}

\begin{lem}\label{reabadeq}
For any formula $A$, there exists an adequate $A$-removing mapping $\mathsf{re}_A$.
\end{lem}
\begin{proof}

We define $\mathsf{re}_A$ by induction on the structure of the formula $A $.



Case 1: $A$ has the form $p$. In this case, $\mathsf{re}_{p} $ is defined by Lemma \ref{repadeq}.
 
Case 2: $A$ has the form $\bot$. Then we put $\mathsf{re}_{\bot}(\pi,\tau):= \mathsf{i}_\bot (\pi)$, where $\mathsf{i}_\bot$ is a non-expansive adequate mapping from Lemma \ref{inversion}.

Case 3: $A$ has the form $B\to C$. Then we put    
\[\mathsf{re}_{B \to C}(\pi,\tau):=\mathsf{re}_C(\mathsf{re}_B(\mathsf{wk}_{\emptyset, C}(\mathsf{ri}_{B\to C}(\tau)),\mathsf{i}_{B\to C}(\pi)),\mathsf{li}_{B\to C}(\tau))\,\]
where $\mathsf{ri}_{B \to C}$, $\mathsf{i}_{B \to C}$, $\mathsf{li}_{B \to C}$ are non-expansive adeqate mappings from Lemma \ref{inversion} and $\mathsf{wk}_{\emptyset, C}$ is a non-expansive adequate mapping from Lemma \ref{strongweakening}.

Case 4: $A$ has the form $\Box B$. By the induction hypothesis, there is an adequate $B$-removing mapping $\mathsf{re}_B$. The required $\Box B$-removing mapping $\mathsf{re}_{\Box B}$ exists by Lemma \ref{reboxadeq}.
\end{proof}

A mapping $\mathsf{u}\colon \mathcal P \to\mathcal P$ is called \emph{root-preserving} if it maps $\infty$-proofs to $\infty$-proofs of the same sequents. Let $\mathcal{T}$ denote the set of all root-preserving non-expansive mappings from $\mathcal P$ to $\mathcal P$. 

\begin{lem}\label{Comp T}
The pair $(\mathcal T, l_1)$ is a non-empty spherically complete ultrametric space.
\end{lem}
\begin{proof}
The proof of spherical completeness of the space $(\mathcal T, l_1)$ is analogous to the proof of Proposition \ref{SphCom}.
The space is obviously non-empty, since the identity function lies in $\mathcal T$.
\end{proof}



\begin{thm}[cut-elimination]
\label{infcuttoinf}
If $\mathsf{Grz_\infty} +\mathsf{cut}\vdash\Gamma\Rightarrow\Delta$, then $\mathsf{Grz_\infty}\vdash\Gamma\Rightarrow\Delta$.
\end{thm}
\begin{proof}
We obtain the required cut-elimination mapping $\mathsf{ce}$ as the fixed-point of a contractive operator $\mathsf F \colon \mathcal T \to T$. 

For a mapping $\mathsf u\in \mathcal T$ and an $\infty$-proof $\pi$, the $\infty$-proof $\mathsf F(\mathsf u)(\pi)$ is defined as follows. If $\lvert \pi\rvert=0$, then we define $\mathsf F(\mathsf u)(\pi)$ to be $\pi$. 




Otherwise, we define $\mathsf F(\mathsf u)(\pi)$ according to the last application of an inference rule in $\pi$:
\begin{gather*}
\AXC{$\pi_1$}
\noLine
\UIC{$\Gamma , B \Rightarrow  \Delta$}
\AXC{$\pi_2$}
\noLine
\UIC{$\Gamma \Rightarrow  A, \Delta$}
\LeftLabel{$\mathsf{\rightarrow_L}$}
\BIC{$\Gamma , A \rightarrow B \Rightarrow  \Delta$}
\DisplayProof 
\longmapsto
\AXC{$\mathsf u(\pi_1)$}
\noLine
\UIC{$\Gamma , B \Rightarrow  \Delta$}
\AXC{$\mathsf u(\pi_2)$}
\noLine
\UIC{$\Gamma \Rightarrow  A, \Delta$}
\LeftLabel{$\mathsf{\rightarrow_L}$}
\RightLabel{ ,}
\BIC{$\Gamma , A \rightarrow B \Rightarrow  \Delta$}
\DisplayProof 
\end{gather*}
\begin{gather*}
\AXC{$\pi_0$}
\noLine
\UIC{$\Gamma, A \Rightarrow  B , \Delta$}
\LeftLabel{$\mathsf{\rightarrow_R}$}
\UIC{$\Gamma \Rightarrow  A \rightarrow B , \Delta$}
\DisplayProof 
\longmapsto
\AXC{$\mathsf u(\pi_0)$}
\noLine
\UIC{$\Gamma, A \Rightarrow  B , \Delta$}
\LeftLabel{$\mathsf{\rightarrow_R}$}
\RightLabel{ ,}
\UIC{$\Gamma \Rightarrow  A \rightarrow B , \Delta$}
\DisplayProof 
\end{gather*}
\begin{gather*}
\AXC{$\pi_0$}
\noLine
\UIC{$\Gamma, A, \Box A \Rightarrow   \Delta$}
\LeftLabel{$\mathsf{refl}$}
\UIC{$\Gamma , \Box A\Rightarrow  \Delta$}
\DisplayProof 
\longmapsto
\AXC{$\mathsf u(\pi_0)$}
\noLine
\UIC{$\Gamma, A, \Box A \Rightarrow   \Delta$}
\LeftLabel{$\mathsf{refl}$}
\RightLabel{ ,}
\UIC{$\Gamma , \Box A\Rightarrow  \Delta$}
\DisplayProof 
\end{gather*}
\begin{align*}
\AXC{$\pi_1$}
\noLine
\UIC{$\Gamma, \Box \Pi \Rightarrow A, \Delta$}
\AXC{$\pi_2$}
\noLine
\UIC{$\Box \Pi \Rightarrow A$}
\LeftLabel{$\mathsf{\Box}$}
\BIC{$\Gamma, \Box \Pi \Rightarrow \Box A, \Delta$}
\DisplayProof 
&\longmapsto
\AXC{$\mathsf u(\pi_1)$}
\noLine
\UIC{$\Gamma, \Box \Pi \Rightarrow A, \Delta$}
\AXC{$\mathsf{u}(\pi_2)$}
\noLine
\UIC{$\Box \Pi \Rightarrow A$}
\LeftLabel{$\mathsf{\Box}$}
\RightLabel{ .}
\BIC{$\Gamma, \Box \Pi \Rightarrow \Box A, \Delta$}
\DisplayProof\\\\
\AXC{$\pi_1$}
\noLine
\UIC{$\Gamma \Rightarrow A, \Delta$}
\AXC{$\pi_2$}
\noLine
\UIC{$\Gamma, A \Rightarrow \Delta$}
\LeftLabel{$\mathsf{cut}$}
\BIC{$\Gamma\Rightarrow  \Delta$}
\DisplayProof 
&\longmapsto
\mathsf {re}_A(\mathsf u(\pi_0),\mathsf u(\pi_1)).
\end{align*}
Now the operator $\mathsf F$ is well-defined.
By the case analysis according to the definition of $\mathsf F$, we see that $\mathsf F (\mathsf u) $ is non-expansive and belongs to $\mathcal T$ whenever $\mathsf u \in \mathcal T$.

We claim that $\mathsf {F}$ is contractive. It sufficient to check that for any $\mathsf u, \mathsf v\in \mathcal T$ and any $n,k\in\mathbb N$ we have 
\[\mathsf u\sim_{n,k}\mathsf v \Rightarrow \mathsf {F(u)}\sim_{n,k+1}\mathsf{F(v)}.\]

Assume there are mappings $\mathsf u$ and $\mathsf v$ from $\mathcal{T}$ such that $\mathsf u\sim_{n,k}\mathsf v$. Consider an arbitrary $\infty$-proof $\pi$. By the case analysis according to the definition of $\mathsf {F}$, we prove that $\mathsf F(\mathsf u)(\pi)\sim_{n}\mathsf F(\mathsf v)(\pi)$. In addition, we check that if $\lvert\pi\rvert<k+1$, then $\mathsf F(\mathsf u)(\pi)\sim_{n+1}\mathsf F(\mathsf v)(\pi)$. 

Let us consider the case when $\pi$ has the form
\[\AXC{$\pi_1$}
\noLine
\UIC{$\Gamma, \Box \Pi \Rightarrow A, \Delta$}
\AXC{$\pi_2$}
\noLine
\UIC{$\Box \Pi \Rightarrow A$}
\LeftLabel{$\mathsf{\Box}$}
\RightLabel{ .}
\BIC{$\Gamma, \Box \Pi \Rightarrow \Box A, \Delta$}
\DisplayProof 
\]
We have $\mathsf u(\pi_0) \sim_n \mathsf v(\pi_0)$ and  $\mathsf u(\pi_1) \sim_n \mathsf v(\pi_1)$. Thus, $F(\mathsf u)(\pi) \sim_n F(\mathsf v)(\pi)$.

If $\lvert\pi\rvert<k+1$, then $\lvert\pi_0\rvert<k$.
We have $\mathsf u(\pi_0) \sim_{n+1} \mathsf v(\pi_0)$ and  $\mathsf u(\pi_1) \sim_n \mathsf v(\pi_1)$.
Hence, $F(\mathsf u)(\pi) \sim_{n+1} F(\mathsf v)(\pi)$.

Now consider the case when $\pi$ has the form
\[\AXC{$\pi_1$}
\noLine
\UIC{$\Gamma \Rightarrow A, \Delta$}
\AXC{$\pi_2$}
\noLine
\UIC{$\Gamma, A \Rightarrow \Delta$}
\LeftLabel{$\mathsf{cut}$}
\RightLabel{ .}
\BIC{$\Gamma\Rightarrow  \Delta$}
\DisplayProof 
\] 
We have $\mathsf u(\pi_0) \sim_n \mathsf v(\pi_0)$ and  $\mathsf u(\pi_1) \sim_n \mathsf v(\pi_1)$.
Since the mapping $\mathsf {re}_A$ is non-expansive, we see that
\[F(\mathsf u)(\pi) = \mathsf {re}_A(\mathsf u(\pi_0),\mathsf u(\pi_1))
\sim_n \mathsf {re}_A(\mathsf v(\pi_0),\mathsf v(\pi_1))= F(\mathsf v)(\pi).\]

Moreover, if $\lvert\pi\rvert<k+1$, then $\lvert\pi_0\rvert<k$ and $\lvert\pi_1\rvert<k$. We obtain $\mathsf u(\pi_0) \sim_{n+1} \mathsf v(\pi_0)$ and  $\mathsf u(\pi_1) \sim_{n+1} \mathsf v(\pi_1)$. Hence,
\[F(\mathsf u)(\pi) = \mathsf {re}_A(\mathsf u(\pi_0),\mathsf u(\pi_1))
\sim_{n+1} \mathsf {re}_A(\mathsf v(\pi_0),\mathsf v(\pi_1))= F(\mathsf v)(\pi).\]

Other cases are straightforward, so we omit them. We have that the operator $\mathsf {F}$ is contractive. 

Now we define the required cut-elimination mapping $\mathsf{ce}$ as the fixed-point of the the operator $\mathsf F \colon \mathcal T \to \mathcal T$, which exists by Lemma \ref{Comp T} and Theorem \ref{fixpoint}.

For some $n,k\in\mathbb N$, let us call a mapping $\mathsf u\in \mathcal T$ $(n,k)$-free if it satisfies the following two conditions: $\mathsf u(\pi)\in\mathcal P_n$ for any $\pi \in \mathcal P$; $\mathsf u(\pi)\in\mathcal P_{n+1}$ whenever $\lvert \pi\rvert<k$. 


We claim that $F(\mathsf u)$ is $(n,k+1)$-free whenever $\mathsf{u}$ is a $(n,k)$-free mapping from $\mathcal T$. Consider an arbitrary $\infty$-proof $\pi$ and an $(n,k)$-free mapping $\mathsf u\in\mathcal T$. By the case analysis according to the definition of $\mathsf {F}$, we prove that $F(\mathsf u)(\pi)\in \mathcal{P}_n$ for any $\pi \in \mathcal P$. In addition, we check that $\mathsf F(\mathsf u)(\pi)\in \mathcal P_{n+1}$ if $\lvert\pi\rvert<k+1$. 

Let us consider the case when $\pi$ has the form
\[\AXC{$\pi_1$}
\noLine
\UIC{$\Gamma, \Box \Pi \Rightarrow A, \Delta$}
\AXC{$\pi_2$}
\noLine
\UIC{$\Box \Pi \Rightarrow A$}
\LeftLabel{$\mathsf{\Box}$}
\RightLabel{ .}
\BIC{$\Gamma, \Box \Pi \Rightarrow \Box A, \Delta$}
\DisplayProof 
\]
We have $\mathsf u(\pi_0) ,\mathsf u(\pi_1) \in \mathcal{P}_n$. Hence, $F(\mathsf u)(\pi) \in \mathcal{P}_n$.

If $\lvert\pi\rvert<k+1$, then $\lvert\pi_0\rvert<k$. We have $\mathsf u(\pi_0) \in \mathcal{P}_{n+1}$ and $\mathsf u(\pi_1) \in \mathcal{P}_n$. Hence, $F(\mathsf u)(\pi) \in \mathcal{P}_{n+1}$.

Now we consider the case when $\pi$ has the form
\[\AXC{$\pi_1$}
\noLine
\UIC{$\Gamma \Rightarrow A, \Delta$}
\AXC{$\pi_2$}
\noLine
\UIC{$\Gamma, A \Rightarrow \Delta$}
\LeftLabel{$\mathsf{cut}$}
\RightLabel{ .}
\BIC{$\Gamma\Rightarrow  \Delta$}
\DisplayProof 
\] 
We have $\mathsf u(\pi_0) ,\mathsf u(\pi_1) \in \mathcal{P}_n$. Since the mapping $\mathsf {re}_A$ is adequate, we see that $F(\mathsf u)(\pi) = \mathsf {re}_A(\mathsf u(\pi_0),\mathsf u(\pi_1)) \in \mathcal{P}_n$.

If $\lvert\pi\rvert<k+1$, then $\lvert\pi_0\rvert<k$ and $\lvert\pi_1\rvert<k$. We obtain $\mathsf u(\pi_0), \mathsf{u}(\pi_1) \in \mathcal{P}_{n+1}$. Hence, $F(\mathsf u)(\pi) = \mathsf {re}_A(\mathsf u(\pi_0),\mathsf u(\pi_1))\in \mathcal{P}_{n+1}$.

We omit other cases of lowermost inferences in $\pi$, because they are trivial.


We established that $\mathsf F(\mathsf u)$ is $(n,k+1)$-free for any $(n,k)$-free $\mathsf u\in\mathcal T$. Notice that if a mapping $\mathsf{u} $ is $(n,k)$-free for all $k\in\mathbb N$, then it is also $(n+1,0)$-free. Now by induction on $n$ with a subinduction on $k$, we immediately obtain that the mapping $\mathsf {ce}$, which is a fixed-point of $\mathsf F$, is $(n,k)$-free for all $n,k\in\mathbb N$. Therefore, for any $\infty$-proof $\pi$, the $\infty$-proof $\mathsf {ce}(\pi)$ does not contain instances of the rule $\mathsf {(cut)}$.



Now assume $\mathsf{Grz_\infty} +\mathsf{cut}\vdash\Gamma\Rightarrow\Delta$. Take an $\infty$-proof of the sequent $\Gamma\Rightarrow\Delta$ in the system $\mathsf{Grz_\infty} +\mathsf{cut}$ and apply the mapping $\mathsf{ce}$ to it. We obtain an $\infty$-proof of the same sequent in the system $\mathsf{Grz_\infty}$.
\end{proof}
Theorem \ref{cutelimgrz} is now established as a direct consequence of Theorem \ref{seqtoinfcut}, Theorem \ref{infcuttoinf}, and Theorem \ref{inftoseq}.

%
\section{Lyndon interpolation}
\label{SecLyn}
The Lyndon interpolation property for the logic was established in \cite{Maks} on the basis of Kripke semantics. Here we present a proof-theoretic argument for the same result. 

For a formula or a set of formulas $X$, let $\mathsf{pos}(X)$ be the set of atomic propositions that have positive occurrences in $X$ and let $\mathsf{neg}(X)$ be the set of atomic propositions with negative occurrences in $X$. For a sequent $\Gamma\Rightarrow\Delta$, let $\mathit{pos}(\Gamma\Rightarrow\Delta)=\mathit{pos}(\Delta)\cup \mathit{neg}(\Gamma)$ and $\mathit{neg}(\Gamma\Rightarrow\Delta)=\mathit{neg}(\Delta)\cup \mathit{pos}(\Gamma)$.

Recall that for a finite set of formulas $\Lambda$, the set $\Lambda^\ast$ is $\{\Box(A\to\Box A)\mid A\in\Lambda\}$.

\begin{lem} 
For any finite sets of formulas $\Lambda_1$, $\Lambda_2$, 
if $\mathsf{Grz_\infty}\vdash \Gamma_1,\Gamma_2\Rightarrow\Delta_1,\Delta_2$, then there exists a formula $I$ called an interpolant of this sequent, such that $\mathit{pos}(I)\subset \mathit{neg}(\Gamma_1\Rightarrow\Delta_1)\cap \mathit{pos}(\Gamma_2\Rightarrow\Delta_2)$, $\mathit{neg}(I)\subset \mathit{pos}(\Gamma_1\Rightarrow\Delta_1)\cap \mathit{neg}(\Gamma_2\Rightarrow\Delta_2)$, and
$$\mathsf{Grz} \vdash \bigwedge \Lambda_1^\ast \cup\Gamma_1\rightarrow\bigvee\Delta_1 \cup \{ I\} \quad \text{and} \quad\mathsf{Grz} \vdash \bigwedge \Lambda_2^\ast \cup\{ I\} \cup  \Gamma_2\rightarrow\bigvee\Delta_2\:.$$
\end{lem}
\begin{proof}
Assume $\pi$ is an $\infty$--proof of the sequent $\Gamma_1,\Gamma_2\Rightarrow\Delta_1,\Delta_2$ in $\mathsf{Grz}_\infty$ and $\Lambda_1$, $\Lambda_2$ are finite sets of formulas.
We prove the statement of the lemma by induction on the sum of cardinalities $\mathit{card}(Sub(\Gamma_1\Rightarrow\Delta_1)\backslash \Lambda_1 )+\mathit{card}(  Sub(\Gamma_2\Rightarrow\Delta_2)\backslash  \Lambda_2)$ with a subinduction on $\lvert \pi \rvert$.



Suppose $\lvert \pi \rvert=0$. Then $\Gamma_1,\Gamma_2\Rightarrow\Delta_1,\Delta_2$ is an initial sequent. 
If $\Gamma_1\Rightarrow\Delta_1$ ($\Gamma_2\Rightarrow\Delta_2$) is an initial sequent, then we put $I:= \bot$ ($I:= \top$). Otherwise, we have an atomic proposition $p$ such that $p \in \Gamma_1$ and $p\in \Delta_2$ ($p \in \Gamma_2$ and $p\in \Delta_1$). In this case, we put $I:= p$ ($I:= \neg p$).

Now consider the last application of an inference rule in $\pi$. 

Case 1. Suppose that $\pi$ has the form
\begin{gather*}
\AXC{$\pi^\prime$}
\noLine
\UIC{$\Gamma_1,\Gamma_2,A\Rightarrow B,\Sigma$}
\LeftLabel{$\mathsf{\to_R}$}
\RightLabel{ ,}
\UIC{$\Gamma_1,\Gamma_2\Rightarrow A\to B,\Sigma$}
\DisplayProof
\end{gather*}
where $A\to B,\Sigma = \Delta_1, \Delta_2$.
Notice that $\lvert \pi^\prime \rvert < \lvert \pi \rvert $. 
If $A \to B \in \Delta_1$ ($A \to B \in \Delta_2$), then we put $\Gamma^\prime_1:= \Gamma_1 \cup \{A\}$, $\Delta^\prime_1:= (\Delta_1\backslash \{A\to B\})\cup \{B\}$, $\Gamma^\prime_2:=\Gamma_2$, $\Delta^\prime_2:=\Delta_2$ ($\Gamma^\prime_2:= \Gamma_2 \cup\{A\}$, $\Delta^\prime_2:= (\Delta_2\backslash \{A\to B\})\cup\{B\}$, $\Gamma^\prime_1:=\Gamma_1$, $\Delta^\prime_1:=\Delta_1$). We see that $\pi^\prime$ is an $\infty$-proof of the sequent $\Gamma^\prime_1,\Gamma^\prime_2\Rightarrow\Delta^\prime_1,\Delta^\prime_2$ in $\mathsf{Grz}_\infty$. By the induction hypothesis for $\pi^\prime$, $\Lambda_1$ and $\Lambda_2$, there is the corresponding interpolant $I^\prime$. We set $I:=I^\prime$.

Case 2. Suppose that $\pi$ has the form
\begin{gather*}
\AXC{$\pi^\prime$}
\noLine
\UIC{$\Sigma, B\Rightarrow \Delta_1,\Delta_2$}
\AXC{$\pi^{\prime\prime}$}
\noLine
\UIC{$\Sigma \Rightarrow A,\Delta_1,\Delta_2$}
\LeftLabel{$\mathsf{\to_L}$}
\RightLabel{ ,}
\BIC{$\Sigma, A\to B\Rightarrow \Delta_1,\Delta_2$}
\DisplayProof
\end{gather*}
where $\Sigma, A\to B = \Gamma_1,\Gamma_2$. We see that $\lvert \pi^\prime \rvert < \lvert \pi \rvert $. 
If $A \to B \in \Gamma_1$ ($A \to B \in \Gamma_2$), then we put $\Gamma^\prime_1:= (\Gamma_1  \backslash \{A\to B\})\cup \{B\}$, $\Gamma^{\prime\prime}_1:= \Gamma_1  \backslash \{A\to B\}$, 
$\Delta^\prime_1:= \Delta_1  $, $\Delta^{\prime\prime}_1:= \Delta_1    \cup \{A\}$, $\Gamma^\prime_2:=\Gamma_2$, $\Gamma^{\prime\prime}_2:=\Gamma_2$,
$\Delta^\prime_2:=\Delta_2$, $\Delta^{\prime\prime}_2:=\Delta_2$
($\Gamma^\prime_2:= (\Gamma_2  \backslash \{A\to B\})\cup \{B\}$, $\Gamma^{\prime\prime}_2:= \Gamma_2  \backslash \{A\to B\}$, 
$\Delta^\prime_2:= \Delta_2  $, $\Delta^{\prime\prime}_2:= \Delta_2    \cup \{A\}$, 
$\Gamma^\prime_1:=\Gamma_1$, $\Gamma^{\prime\prime}_1:=\Gamma_1$,
$\Delta^\prime_1:=\Delta_1$, $\Delta^{\prime\prime}_1:=\Delta_1$). We see that $\pi^\prime$ and $\pi^{\prime\prime}$ are $\infty$-proofs of $\Gamma^\prime_1,\Gamma^\prime_2\Rightarrow\Delta^\prime_1,\Delta^\prime_2$ and $\Gamma^{\prime\prime}_1,\Gamma^{\prime\prime}_2\Rightarrow\Delta^{\prime\prime}_1,\Delta^{\prime\prime}_2$ in $\mathsf{Grz}_\infty$. By the induction hypothesis for $\pi^\prime$, $\Lambda_1$ and $\Lambda_2$, there is a interpolant $I^\prime$. Analogously, by the induction hypothesis for $\pi^{\prime\prime}$, $\Lambda_1$ and $\Lambda_2$, there is a interpolant 
$I^{\prime\prime}$. If $A \to B \in \Gamma_1$, then we set $I:=I^\prime \vee I^{\prime\prime}$. Otherwise, when $A \to B \in \Gamma_2$, we set $I:=I^\prime \wedge I^{\prime\prime}$.

Case 3. Suppose that $\pi$ has the form
\begin{gather*}
\AXC{$\pi^\prime$}
\noLine
\UIC{$\Sigma,A,\Box A\Rightarrow \Delta_1,\Delta_2$}
\LeftLabel{$\mathsf{refl}$}
\RightLabel{ ,}
\UIC{$\Sigma,\Box A\Rightarrow \Delta_1,\Delta_2$}
\DisplayProof
\end{gather*}
where $\Sigma, \Box A = \Gamma_1,\Gamma_2$. If $\Box A \in \Gamma_1$ ($\Box A \in \Gamma_2$), then we put $\Gamma^\prime_1:= \Gamma_1 \cup \{A\}$, $\Delta^\prime_1:= \Delta_1$, $\Gamma^\prime_2:=\Gamma_2$, $\Delta^\prime_2:=\Delta_2$ ($\Gamma^\prime_2:= \Gamma_2 \cup\{A\}$, $\Delta^\prime_2:= \Delta_2$, $\Gamma^\prime_1:=\Gamma_1$, $\Delta^\prime_1:=\Delta_1$). We see that $\pi^\prime$ is an $\infty$-proof of the sequent $\Gamma^\prime_1,\Gamma^\prime_2\Rightarrow\Delta^\prime_1,\Delta^\prime_2$ in $\mathsf{Grz}_\infty$. By the induction hypothesis for $\pi^\prime$, $\Lambda_1$ and $\Lambda_2$, there is the corresponding interpolant $I^\prime$. We set $I:=I^\prime$.
 
Case 4. Suppose that $\pi$ has the form
\begin{gather*}
\AXC{$\pi^\prime$}
\noLine
\UIC{$\Phi, \Box \Pi \Rightarrow A, \Sigma$}
\AXC{$\pi^{\prime\prime}$}
\noLine
\UIC{$\Box \Pi \Rightarrow A$}
\LeftLabel{$\mathsf{\Box}$}
\RightLabel{ ,}
\BIC{$\Phi, \Box \Pi \Rightarrow \Box A, \Sigma$}
\DisplayProof
\end{gather*}
where $\Phi, \Box \Pi = \Gamma_1, \Gamma_2$ and $\Box A, \Sigma =\Delta_1, \Delta_2$.

Subcase 4.1: the formula $\Box A \in \Delta_1$ and $A \in \Lambda_1 $ ($\Box A \in \Delta_2$ and $A \in \Lambda_2$). We see that $\lvert \pi^\prime \rvert < \lvert \pi \rvert $. 
We put $\Gamma^\prime_1:= \Gamma_1 $, $\Delta^\prime_1:= (\Delta_1 \backslash \{\Box A\}) \cup \{A\} $, $\Gamma^\prime_2:=\Gamma_2$, $\Delta^\prime_2:=\Delta_2$ ($\Gamma^\prime_2:= \Gamma_2 $, $\Delta^\prime_2:= (\Delta_2 \backslash \{\Box A\}) \cup \{A\} $, $\Gamma^\prime_1:=\Gamma_1$, $\Delta^\prime_1:=\Delta_1$). We see that $\pi^\prime$ is an $\infty$-proof of the sequent $\Gamma^\prime_1,\Gamma^\prime_2\Rightarrow\Delta^\prime_1,\Delta^\prime_2$ in $\mathsf{Grz}_\infty$. By the induction hypothesis for $\pi^\prime$, $\Lambda_1$ and $\Lambda_2$, there is the corresponding interpolant $I^\prime$ such that
$$\mathsf{Grz} \vdash \bigwedge \Lambda_1^\ast \cup\Gamma^\prime_1\rightarrow\bigvee\Delta^\prime_1 \cup \{ I^\prime\} \quad \text{and} \quad\mathsf{Grz} \vdash \bigwedge \Lambda_2^\ast \cup\{ I^\prime\} \cup  \Gamma^\prime_2\rightarrow\bigvee\Delta^\prime_2\:.$$
If $\Box A \in \Delta_1$ and $A \in \Lambda_1 $, then $\Box (A \to \Box A) \in \Lambda_1^\ast$. Thus we have $\mathsf{Grz} \vdash \bigwedge \Lambda_1^\ast \rightarrow (A \to \Box A)$. It follows that we can replace the formula $A$ in $\Delta^\prime_1$ by $\Box A$ and obtain 
$$\mathsf{Grz} \vdash \bigwedge \Lambda_1^\ast \cup\Gamma^\prime_1\rightarrow\bigvee\Delta_1 \cup \{ I^\prime \} \:.$$
Note also that $\Gamma^\prime_1:= \Gamma_1 $, $\Gamma^\prime_2:=\Gamma_2$, $\Delta^\prime_2:=\Delta_2$. Hence, 
$$\mathsf{Grz} \vdash \bigwedge \Lambda_1^\ast \cup\Gamma_1\rightarrow\bigvee\Delta_1 \cup \{ I^\prime \}\quad \text{and} \quad\mathsf{Grz} \vdash \bigwedge \Lambda_2^\ast \cup\{ I^\prime\} \cup  \Gamma_2\rightarrow\bigvee\Delta_2\:.$$
We set $I:=I^\prime$. The case  $\Box A \in \Delta_2$ and $A \in \Lambda_2 $ is analogous. 

Subcase 4.2: the formula $\Box A \in \Delta_1$ and $A \nin \Lambda_1 $ ($\Box A \in \Delta_2$ and $A \nin \Lambda_2$). We have that $\Gamma_1, \Gamma_2 = \Phi, \Box \Pi$ and $\Delta_1, \Delta_2 = \Box A, \Sigma$. We split $\Box \Pi$ into two multisets $\Box \Pi_1 \subset \Gamma_1$ and $\Box \Pi_1 \subset \Gamma_2$ such that $\Box \Pi= \Box \Pi_1, \Box \Pi_2$. Put $\Gamma^{\prime\prime}_1:= \Box \Pi_1 $, $\Delta^{\prime\prime}_1:=  \{A\} $, $\Gamma^{\prime\prime}_2:=\Box \Pi_2$, $\Delta^{\prime\prime}_2:=\emptyset$ ($\Gamma^{\prime\prime}_1:= \Box \Pi_1 $, $\Delta^{\prime\prime}_1:=  \emptyset $, $\Gamma^{\prime\prime}_2:=\Box \Pi_2$, $\Delta^{\prime\prime}_2:=\{A\}$).
We see that $\pi^{\prime\prime}$ is an $\infty$-proof of the sequent $\Gamma^{\prime\prime}_1,\Gamma^{\prime\prime}_2\Rightarrow\Delta^{\prime\prime}_1,\Delta^{\prime\prime}_2$ in $\mathsf{Grz}_\infty$.

If $\Box A \in \Delta_2$ and $A \nin \Lambda_2 $, then 
$\mathit{card}(Sub(\Box\Pi_1\Rightarrow)\backslash \Lambda_1 )+\mathit{card}(  Sub(\Box\Pi_2\Rightarrow A)\backslash  (\Lambda_2\cup \{A\})$ is strictly less than
$\mathit{card}(Sub(\Gamma_1\Rightarrow\Delta_1)\backslash \Lambda_1 )+\mathit{card}(  Sub(\Gamma_2\Rightarrow\Delta_2)\backslash  \Lambda_2)$. Hence by the induction hypothesis for $\pi^{\prime\prime}$, $\Lambda_1$ and $\Lambda_2\cup \{A\}$ 
there exists an interpolant $I^{\prime\prime}$ such that
$$\mathsf{Grz} \vdash \bigwedge \Lambda_1^\ast \cup\Box \Pi_1\rightarrow  I^{\prime\prime} \quad \text{and} \quad\mathsf{Grz} \vdash \bigwedge \Lambda_2^\ast \cup \{\Box (A \to \Box A)\}\cup\{ I^{\prime\prime}\} \cup  \Box \Pi_2\rightarrow A\:.$$

From the left condition we immediately obtain $$\mathsf{Grz} \vdash \bigwedge \Lambda_1^\ast \cup\Box \Pi_1\rightarrow  \Box I^{\prime\prime} \quad \text{and} \quad \mathsf{Grz} \vdash \bigwedge \Lambda_1^\ast \cup\Gamma_1\rightarrow  \bigvee\Delta_1 \cup\{\Box  I^{\prime\prime} \}\:.$$

From the right condition we see $$\mathsf{Grz} \vdash \bigwedge \Lambda_2^\ast  \cup\{ I^{\prime\prime}\} \cup  \Box \Pi_2\rightarrow (\Box (A \to \Box A) \to A)\:.$$
Thus, 
$$\mathsf{Grz} \vdash \bigwedge \Lambda_2^\ast  \cup\{ \Box I^{\prime\prime}\} \cup  \Box \Pi_2\rightarrow \Box (\Box (A \to \Box A) \to A)\:.$$
Applying the axiom (v) of $\mathsf{Grz} $ we see that 
$$\mathsf{Grz} \vdash \bigwedge \Lambda_2^\ast  \cup\{ \Box I^{\prime\prime}\} \cup  \Box \Pi_2\rightarrow \Box A\quad \text{and} \quad \mathsf{Grz} \vdash \bigwedge \Lambda_2^\ast  \cup\{ \Box I^{\prime\prime}\} \cup  \Gamma_2\rightarrow \bigvee\Delta_2 \:.$$
Now we can set $I:=\Box I^{\prime\prime}$.
In the case $\Box A \in \Delta_2$, $A \nin \Lambda_2$ we can analogously set $I:=\Diamond I^{\prime\prime}$.

\end{proof}

\begin{thm}[Lyndon interpolation] 
If $\mathsf{Grz} \vdash A \rightarrow B$, then there exists a formula $C$, called an interpolant of $A \rightarrow B$, such that $\mathit{pos} (C) \subset \mathit{pos} (A) \cap \mathit{pos} (B)$, $\mathit{neg} (C) \subset \mathit{neg} (A) \cap \mathit{neg} (B)$, and \[\mathsf{Grz }\vdash A \rightarrow C , \qquad \mathsf{Grz} \vdash C \rightarrow B.\] 
\end{thm} 
\begin{proof}
Assume $\mathsf{Grz} \vdash A \rightarrow B$. By Lemma \ref{prop}, we have $\mathsf{Grz_{Seq}} + \mathsf{cut} \vdash A \Rightarrow B$. Applying Theorem \ref{seqtoinfcut} and Theorem \ref{infcuttoinf} we obtain $\mathsf{Grz}_\infty \vdash A \Rightarrow B$. Applying the previous lemma with $\Lambda_1=\Lambda_2=\emptyset$, $\Gamma_1= \{ A\}$, $\Delta_2=\{ B\}$, $\Gamma_2=\Delta_1=\emptyset $, we find an interpolant for $A \to B$.    
\end{proof}

\section{Cyclic proofs}
\label{SecCyc}

There exists a simple class of $\infty$-proofs that is sufficient to derive all theorems of $\mathsf{Grz}_\infty$. An $\infty$-proof is called \emph{regular} if it contains only finitely many nonisomorphic subtrees.

Regular $\infty$-proofs have useful finite representations called cyclic (circular) proofs.
A \emph{cyclic proof} of a sequent $\Gamma \Rightarrow \Delta$ is a pair $(\kappa, d)$,
where $\kappa$ is a finite tree of sequents constructed according to the rules of $\mathsf{Grz}_\infty$ with
the root marked by $\Gamma \Rightarrow \Delta$, and $d$ is a function with the following properties: the
function $d$ is defined on the set of all leaves of $\kappa$ that are not marked by initial
sequents; the image $d(a)$ of a leaf $a$ lies on the path from the root of $\kappa$ to the leaf $a$; there is a right premise of the rule $(\Box)$ between $a$ and $d(a)$; $a$ and $d(a)$ are marked by the same sequents. If the function $d$
is defined at a leaf $a$, then we say that the nodes $a$ and $d(a)$ are joined by a back-link.
Here is an example of a cyclic proof for the sequent $\Box(\Box(p \rightarrow \Box p) \rightarrow p) \Rightarrow p$: 

\begin{gather*}
\AXC{\textsf{Ax}}
\noLine
\UIC{$ F, p\Rightarrow p$}
\AXC{\textsf{Ax}}
\noLine
\UIC{$ F,p\Rightarrow \Box p,p$}
\LeftLabel{$\mathsf{\to_R}$}
\UIC{$ F \Rightarrow p\to\Box p,p$}
\AXC{\textsf{Ax}}
\noLine
\UIC{$p, F \Rightarrow p$}
\AXC{$ F \Rightarrow p $ \tikzmark{a}}
\LeftLabel{$\mathsf{\Box}$} 
\BIC{$p, F \Rightarrow \Box p$}
\LeftLabel{$\mathsf{}\to_R$}
\UIC{$ F \Rightarrow p\to\Box p$}
\LeftLabel{$\mathsf{\Box}$} 
\BIC{$ F \Rightarrow \Box(p\to \Box p),p$} 
\LeftLabel{$\mathsf{\to_L}$}
\BIC{$\Box(p \rightarrow \Box p) \rightarrow p, F \Rightarrow p$}
\LeftLabel{$\mathsf{refl}$}
\RightLabel{ ,} 
\UIC{$F \Rightarrow p $ \tikzmark{b}}
\DisplayProof
\begin{tikzpicture}[overlay,remember picture, >=latex,distance=4cm]
    \draw[->, thick] ({pic cs:a}) to [out=0,in=-20]({pic cs:b});
 \end{tikzpicture}
\end{gather*}
where $F=\Box(\Box(p \rightarrow \Box p) \rightarrow p) $. 

The notion of cyclic proof determines the same provability
relation as the notion of regular $\infty$-proof. Obviously, each cyclic proof can be
unravelled into a regular $\infty$-proof. The converse is also true.
\begin{prop}[cf. \cite{Sham}, Proposition 3.1]
Any regular $\infty$-proof of $\mathsf{Grz}_\infty$ can be obtained by unraveling a cyclic proof.
\end{prop}



In the rest of the section we establish that any sequent provable in $\mathsf{Grz}_\infty$ has a cyclic proof.


\begin{lem}\label{contraction}
For any formula $A$, the rules
\begin{gather*}
\AXC{$\Gamma , A,A \Rightarrow  \Delta$}
\LeftLabel{$\mathsf{cl}_{A}$}
\UIC{$\Gamma ,A \Rightarrow  \Delta$}
\DisplayProof\qquad
\AXC{$\Gamma \Rightarrow A,A, \Delta$}
\LeftLabel{$\mathsf{cr}_{A}$}
\UIC{$\Gamma \Rightarrow A, \Delta$}
\DisplayProof
\end{gather*}
are admissible in $\mathsf{Grz}_{\infty}$.
\end{lem}
\begin{proof}
Assume $\pi$ and $\tau$ are $\infty$-proofs of the sequents $\Gamma , A,A \Rightarrow  \Delta$ and $\Gamma \Rightarrow A,A, \Delta$ in the system $\mathsf{Grz}_{\infty}$.
Let $\xi$ be an $\infty$-proof of the sequent $\Gamma , A \Rightarrow  A,\Delta$ in $\mathsf{Grz}_{\infty}$, which exists by Lemma \ref{AtoA}.

The required $\infty$-proofs of the sequents $\Gamma , A \Rightarrow  \Delta$ and $\Gamma  \Rightarrow  A,\Delta$ are defined by setting $
\mathsf{cl}_{A}(\pi)=\mathsf{re}_A(\xi, \pi)$ and 
$
\mathsf{cr}_{A}(\tau) =\mathsf{re}_A(\tau,\xi)$, where $\mathsf{re}_A$ is an $A$-removing mapping from Lemma \ref{reabadeq}. Since the mapping $\mathsf{re}_A$ is adequate, the $\infty$-proofs $
\mathsf{cl}_{A}(\pi)$ and $\mathsf{cr}_{A}(\tau)$ do not contain applications the rule $(\mathsf{cut})$.
 
\end{proof}

Let $\mathcal{T}^\ast$ denote the set of all root-preserving mappings from the set of $\infty$-proofs of $\mathsf{Grz}_\infty$ to itself.

\begin{lem}\label{Comp T*}
The pair $(\mathcal T^\ast, l_1)$ is a non-empty spherically complete ultrametric space.
\end{lem}
\begin{proof}
The proof of spherical completeness of the space $(\mathcal T^\ast, l_1)$ is analogous to the proof of Proposition \ref{SphCom}.
The space is obviously non-empty, since the identity function lies in $\mathcal T^\ast$.
\end{proof}

An application of the modal rule $(\Box)$ 
\[\AXC{$\Gamma, \Box \Pi \Rightarrow A, \Delta$}
\AXC{$\Box \Pi \Rightarrow A$}
\LeftLabel{$\mathsf{\Box}$}
\RightLabel{ .}
\BIC{$\Gamma, \Box \Pi \Rightarrow \Box A, \Delta$}
\DisplayProof 
\]
is called \emph{slim} if the multiset $\Pi$ here is a set.  
An $\infty$-proof is called \emph{slim} if every application of the rule $(\Box)$ in it is slim.

\begin{lem}
If $\mathsf{Grz}_\infty\vdash \Gamma \Rightarrow \Delta$, then the sequent $\Gamma \Rightarrow \Delta$ has a slim $\infty$-proof in $\mathsf{Grz}_\infty$.
\end{lem} 
\begin{proof}

We construct a mapping $\mathsf {slim}\in \mathsf T^\ast$ that maps an $\infty$-proof to a slim $\infty$-proof of the same sequent. This mapping is defined as the fixed-point of a contractive operator $\mathsf H\colon \mathcal T^\ast \to \mathcal T^\ast $. 

For a mapping $\mathsf u\in \mathcal T^\ast$ and an $\infty$-proof of $\mathsf{Grz}_\infty$, the $\infty$-proof $\mathsf H(\mathsf u)(\pi)$ is defined as follows.
If $\lvert \pi\rvert=0$, then we put $\mathsf H(\mathsf u)(\pi)=\pi$. 

Otherwise, we define $\mathsf H(\mathsf u)(\pi)$ according to the last application of an inference rule in $\pi$:
\begin{gather*}
\AXC{$\pi_1$}
\noLine
\UIC{$\Gamma , B \Rightarrow  \Delta$}
\AXC{$\pi_2$}
\noLine
\UIC{$\Gamma \Rightarrow  A, \Delta$}
\LeftLabel{$\mathsf{\rightarrow_L}$}
\BIC{$\Gamma , A \rightarrow B \Rightarrow  \Delta$}
\DisplayProof 
\longmapsto
\AXC{$\mathsf u(\pi_1)$}
\noLine
\UIC{$\Gamma , B \Rightarrow  \Delta$}
\AXC{$\mathsf u(\pi_2)$}
\noLine
\UIC{$\Gamma \Rightarrow  A, \Delta$}
\LeftLabel{$\mathsf{\rightarrow_L}$}
\RightLabel{ ,}
\BIC{$\Gamma , A \rightarrow B \Rightarrow  \Delta$}
\DisplayProof 
\end{gather*}
\begin{gather*}
\AXC{$\pi_0$}
\noLine
\UIC{$\Gamma, A \Rightarrow  B , \Delta$}
\LeftLabel{$\mathsf{\rightarrow_R}$}
\UIC{$\Gamma \Rightarrow  A \rightarrow B , \Delta$}
\DisplayProof 
\longmapsto
\AXC{$\mathsf u(\pi_0)$}
\noLine
\UIC{$\Gamma, A \Rightarrow  B , \Delta$}
\LeftLabel{$\mathsf{\rightarrow_R}$}
\RightLabel{ ,}
\UIC{$\Gamma \Rightarrow  A \rightarrow B , \Delta$}
\DisplayProof 
\end{gather*}
\begin{gather*}
\AXC{$\pi_0$}
\noLine
\UIC{$\Gamma, A, \Box A \Rightarrow   \Delta$}
\LeftLabel{$\mathsf{refl}$}
\UIC{$\Gamma , \Box A\Rightarrow  \Delta$}
\DisplayProof 
\longmapsto
\AXC{$\mathsf u(\pi_0)$}
\noLine
\UIC{$\Gamma, A, \Box A \Rightarrow   \Delta$}
\LeftLabel{$\mathsf{refl}$}
\RightLabel{ .}
\UIC{$\Gamma , \Box A\Rightarrow  \Delta$}
\DisplayProof 
\end{gather*}

For a multiset $\Pi$, we denote its underlying set by $\Pi^S$. The case of the modal rule is as follows:
\begin{gather*}
\AXC{$\pi_1$}
\noLine
\UIC{$\Gamma, \Box \Pi \Rightarrow A, \Delta$}
\AXC{$\pi_2$}
\noLine
\UIC{$\Box \Pi \Rightarrow A$}
\LeftLabel{$\mathsf{\Box}$}
\BIC{$\Gamma, \Box \Pi \Rightarrow \Box A, \Delta$}
\DisplayProof 
\longmapsto
\AXC{$\mathsf u(\pi_1)$}
\noLine
\UIC{$\Gamma, \Box \Pi \Rightarrow A, \Delta$}
\AXC{$\mathsf {u}(\pi^\prime_2)$}
\noLine
\UIC{$\Box \Pi^S \Rightarrow A$}
\LeftLabel{$\mathsf{\Box}$}
\RightLabel{ ,}
\BIC{$\Gamma, \Box \Pi \Rightarrow \Box A, \Delta$}
\DisplayProof 
\end{gather*}
where $\pi^\prime_2$ is an $\infty$-proof of the sequent $\Box \Pi^S \Rightarrow A$ in $\mathsf{Grz}_\infty$. This $\infty$-proof exists by the
admissibility of contraction rules obtained in Lemma \ref{contraction}.
 
Now it can be easily shown, that if $\mathsf u\sim_{n,k} \mathsf v$ for some $\mathsf u, \mathsf v\in \mathcal T^\ast$ and $n,k\in\mathbb N$, then $\mathsf {H(u)}\sim_{n,k+1}\mathsf{H(v)}$. Therefore the mapping $\mathsf H$ is a contractive operator on a spherically complete ultrametric space. Thus, it has a fixed-point, which we denote by $\mathsf {slim}$.

In analogous way to the proofs of Lemma \ref{reboxadeq} and Theorem \ref{infcuttoinf}, we see that $\mathsf {slim}(\pi)$ is a slim $\infty$-proof for any $\pi$.
\end{proof}


\begin{thm}
If $\mathsf{Grz}_\infty\vdash \Gamma \Rightarrow \Delta$, then there is a cyclic proof for the given sequent. 
\end{thm}
\begin{proof}
Let $\pi$ be a slim $\infty$-proof of the sequent $\Gamma \Rightarrow \Delta$ obtained in the previous lemma. Note that all formulas from $\pi$ are subformulas of the formulas from $\Gamma \cup \Delta$. Consequently, the $\infty$-proof $\pi$ contains only finitely many different sequents that occur as right premises of the rule $(\Box)$ in it. By $k$, we denote the number of these sequents. 

Let $\xi$ be the $(k+2)$-fragment of the $\infty$-proof $\pi$. Consider any branch $a_0, a_1, \dotsc, a_n$ in $\xi$ connecting the root with a leaf $a_n$ that is not marked by an initial sequent. This branch containes a pair of different nodes $a$ and $b$ determining coinciding right premises of the rule $(\Box)$. Assuming that $b$ is further from the root of $\xi$ than $a$, we cut the branch under consideration at the node $b$ and connect $b$, which has become a leaf, with $a$ by a back-link. Applying
a similar operation to the remaining branches of $\xi$, we turn the $(k+2)$-fragment of $\pi$ into a cyclic proof of the sequent $\Gamma \Rightarrow \Delta$.  


\end{proof}


\paragraph*{Funding.} The article was prepared within the framework of the Basic
Research Program at the National Research University 
Higher School of Economics (HSE) and supported within the framework of a subsidy by the Russian Academic Excellence Project 
'5-100'. This work is also supported by the Russian Foundation for Basic Research, grant 15-01-09218a. 

\paragraph*{Acknowledgements.} The second author heartily thanks his beloved wife Maria Shamkanova for her warm and constant support. He is also indebted to the Lord God Almighty.

\end{document}